\documentclass[12pt]{amsart}
\usepackage{amsmath}
\usepackage{amssymb}
\usepackage{latexsym}

\newtheorem{theorem}{Theorem}[section]
\newtheorem{cor}[theorem]{Corollary}
\newtheorem{lemma}[theorem]{Lemma}

\newtheorem{prop}[theorem]{Proposition}
\newtheorem{defn}[theorem]{Definition}

\newenvironment{proof*}{\vskip 2mm\noindent {}}{\hfill $\Box$ \vskip 2mm}
\numberwithin{equation}{section}
\newcommand{\C}{{\mathbb{C}}}
\newcommand{\D}{{\mathbb{D}}}

\newcommand{\N}{{\mathbb{N}}}
\renewcommand{\P}{{\mathbb{P}}}
\newcommand{\R}{{\mathbb{R}}}

\newcommand{\I}{{\mathcal{I}}}
\newcommand{\J}{{\mathcal{J}}}
\renewcommand{\O}{{\mathcal{O}}}
\newcommand{\eps}{\varepsilon}
\newcommand{\Om}{\Omega}

\newcommand{\id}{{\rm id}}
\newcommand{\Sym}{{\rm Sym}}

\begin{document}

\title[Limits of Green functions]{Limits of multipole pluricomplex Green functions}

\author[Magn\'usson, Rashkovskii, Sigurdsson, Thomas]
{J\'on I. Magn\'usson, Alexander Rashkovskii, Ragnar Sigurdsson, Pascal J. Thomas}

\address{J\'on I. Magn\'usson and Ragnar Sigurdsson\\
Science Institute, University of Iceland\\
Dunhaga 3\\ IS-107 Reykjavik, Iceland}
\email{jim@hi.is, ragnar@hi.is}

\address{Alexander Rashkovskii\\
Faculty of Science and Technology,
University of Stavanger\\
N-4036 Stavanger, Norway}
\email{alexander.rashkovskii@uis.no}

\address{Pascal J. Thomas\\ 
Universit\'e de Toulouse\\ UPS, INSA, UT1, UTM \\
Institut de Math\'e\-matiques de Toulouse\\
F-31062 Toulouse, France} \email{pthomas@math.univ-toulouse.fr}

\begin{abstract}
Let $\Omega$ be a bounded
hyperconvex domain in $\mathbb C^n$, $0 \in \Omega$,
and $S_\varepsilon$ a  family of $N$ poles in
$\Omega$, all tending to $0$ as
$\varepsilon$ tends to $0$.
To each $S_\varepsilon$ we associate 
its vanishing ideal 
$\I_\varepsilon$ and pluricomplex
Green function $G_\varepsilon=G_{\I_\varepsilon}$.\\
Suppose that,
as $\varepsilon$ tends to $0$, $(\I_\varepsilon)_\varepsilon$ converges to $\mathcal I$ (local uniform convergence), and that 
$(G_\varepsilon)_\varepsilon$
converges to $G$, locally uniformly away from $0$; then
$G \ge G_\I$.
If the Hilbert-Samuel 
multiplicity of $\mathcal I$ is strictly larger than its length
(codimension, equal to $N$ here), then $(G_\varepsilon)_\varepsilon$
cannot converge to $G_\I$. Conversely,
if $\I$ is a complete intersection ideal, then  
$(G_\varepsilon)_\varepsilon$
converges to $G_\I$.
We work out the case of three poles. 
\end{abstract}


\thanks{Part of this work was carried out during the stays of
the last-named author in the University of Iceland, and of the 
third-named author in the Universit\'e Paul Sabatier, Toulouse,
made possible by two Partenariats Hubert Curien "Jules Verne" in 2006-2007 (12339SE)
and in 2008-2009 (18980ZB),
financed by the French Minist\`ere
des Affaires Etrang\`eres and the Icelandic Ministry of Education, Science
and Culture.
Another part was done during the stay of the second-named
author in the Universit\'e Paul Sabatier as an invited professor
in June 2009.}

\maketitle
\tableofcontents

\section{Introduction}

\subsection{Definitions}

The aim of this paper is to study convergence of pluricomplex Green
functions with simple logarithmic poles at finitely many points as the
poles tend to a single point.
Pluricomplex Green functions with logarithmic singularities have been
studied by many authors at different levels of generality.  See
e.g.~Demailly \cite{De1},  \cite{Za}, Lempert \cite{Lem}, Lelong \cite{Lel}, 
L\'arusson and Sigurdsson \cite{La-Si} , and
Rashkovskii and Sigurdsson \cite{Ra-Si}.  We restrict our attention to
the case when $\Omega$ is a hyperconvex, bounded, contractible domain
in $\C^n$ containing the origin $0$ and we let $\O(\Omega)$ denote the
space of all holomorphic functions on $\Omega$.  If $\I$ is an ideal
in $\O(\Omega)$, then we let $V(\I)$ denote the zero variety of $\I$,
consisting of all common zeros of the functions in $\I$, and for every 
subset $S$ of $\Omega$ we let $\I(S)$ denote the ideal of all
functions vanishing on $S$.  

We will only consider
ideals $\mathcal I$ in $\mathcal O(\Omega)$, such that
$V(\mathcal I)$ is a finite set.  The elements of $\I$ may be 
defined by local conditions, but
by Cartan's Theorem B, there are finitely many generators $\psi_j \in \mathcal O(\Omega)$ such that
for any $f\in \mathcal I$, there exists $h_j \in \mathcal O(\Omega)$ such that 
$f = \sum_j h_j \psi_j$, see e.g. \cite[Theorem 7.2.9, p. 190]{Ho}. 

\begin{defn}
\label{greenideal}
Let $\mathcal I$ be an ideal of $\Omega$.  For each $a\in \Omega$, let $(\psi_{a,i})_i$ be 
a (local) system of generators. Then
\begin{multline*}
G_{\mathcal I}^\Omega (z) := \sup \big\{ u(z) : u \in PSH_-(\Omega), \\
u(z) \le \max_i \log |\psi_{a,i}| + O(1) , \forall a \in \Omega \big\}.
\end{multline*}
\end{defn}

Note that the condition is meaningful only when $a \in V(\mathcal I)$.
It can be proved that 
$\I\subset\J$ implies $G_{\I} \le G_{\J}$. 
In the special case when $S$ is a finite set in $\Omega$ and
$\mathcal I=\mathcal I(S)$, 
we write $G_{\mathcal I(S)}=G_S$: this case reduces to Pluricomplex Green functions with logarithmic singularities.
 
We know that $G_S$ depends continuously on the poles
when those remain a fixed distance apart \cite{Lel}, and would like to know what happens
 when  $S$
 coalesces to a single point. 
 
The setup of our problem is the following. Let $A$ be a subset of $\C$
with the origin in its closure $\overline A$ and let
$(\I_\eps)_{\eps \in A}$ be a family of ideals in
$\O(\Omega)$.  We are interested in the convergence of
$G_{\I_\eps}$ as $A\ni \eps \to 0$  in general and in
particular for the special case when $\I_\eps =\I(S_\eps)$,
where $S_\eps$ is a set of $N$ distinct points 
$a_1^\eps,\dots,a_N^\eps$ all tending to $0$.

We consider ideals of holomorphic functions as points in
the Douady space, with the attendant topology (see Section 
\ref{doucy} for definitions). 

When convergence is locally uniform,
one inequality always holds between the limit Green function
and the one derived from limits of ideals. 

\begin{prop}
\label{glimunif}
If $G_{\mathcal I_\eps}$ converges uniformly
to $g$ on compact subsets of $\Omega \setminus \{0\}$, and
if  the family 
$(\I_\eps)$ converges,
then
$$
G_{\lim\limits_{\eps\to0}  \mathcal I_\eps} \le \lim\limits_{\eps\to0} G_{\mathcal I_\eps} .
$$
\end{prop}
This is proved in Section \ref{gineq}, after Lemma \ref{boundcv}.

The one-pole 
pluricomplex Green function can be seen as a fundamental solution
to the complex Monge-Amp\`ere operator $(dd^c)^n$, where $d=\partial +
\bar \partial$, $d^c:=\frac1{2\pi i}( \partial - \bar \partial)$. 
The choice of constant in $d^c$ ensures that $(dd^c \log \| \cdot \| )^n=\delta_0$, the Dirac mass at the origin.

When $f\in \mathcal C^2$,  $(dd^c f)^n$ 
is a multiple of the determinant of the complex Hessian matrix.
This operator can be extended to plurisubharmonic functions which are bounded outside of a compact subset of $\Omega$ \cite{Be-Ta}, \cite[Chap. III, \S 3]{De2}.  If $u$ is plurisubharmonic and $(dd^c u)^n=0$ on
an open set $\omega$, we say that $u$ is \emph{maximal plurisubharmonic} on $\omega$. This is equivalent to a form of the maximum principle:
if $\omega_0 \Subset \omega$, 
$v \in PSH(\omega_0)$ and $u\ge v$ on $\partial \omega_0$, then $u\ge v$ on
$\omega_0$.  

\begin{cor}
\label{coridineq}
Under the  hypotheses of Proposition \ref{glimunif},
$$
\left( dd^c \lim_{\eps\to0} G_{\mathcal I_\eps}\right)^n  
\le \left( dd^c G_{\lim\limits_{\eps\to0}  \mathcal I_\eps} \right)^n  .
$$
In particular, since $V(\lim\limits_{\eps\to0}  \mathcal I_\eps)=\{0\}$, then $\lim\limits_{\eps\to0} G_{\mathcal I_\eps}$
is maximal plurisubharmonic outside the origin. 
\end{cor}

This is proved at the end of Section
\ref{gineq}. 

\subsection{Another definition of convergence of ideals}

Since many analysts, including some authors of the present work, are not familiar with the Douady space, we have found it useful to work with a more concrete approach to the convergence of ideals. It also allows us to work with notions of limes superior and inferior when the family of ideals fails to converge. 

\begin{defn}
\label{convid}
We call \emph{lower limit} of $(\I_\eps)_{\eps\in A}$,
denoted by $\liminf\limits_{A\ni\eps \to 0}\I_\eps$,  the set
of all $f\in \O(\Omega)$ such that $f_\eps\to f$ locally
uniformly, i.e.~uniformly on every compact subset of $\Omega$, 
as $A\ni \eps\to 0$, where $f_\eps \in \I_\eps$.

We call \emph{upper limit} of $(\I_\eps)_{\eps\in A}$,
denoted by 
$\limsup\limits_{A\ni\eps \to 0}\I_\eps$,  the subspace
of $\O(\Omega)$ generated by all functions $f$ such that
$f_j\to f$ locally uniformly, as $j\to \infty$ for some
sequence $\eps_j\to 0$ in $A$ and $f_j\in \I_{\eps_j}$.

If they are equal then we say that the family 
$\I_\eps$ converges and write 
$\lim\limits_{A\ni\eps \to 0}\I_\eps$ for
the common value of the upper and lower limits. 
\end{defn}

Of course $\liminf\limits_{A\ni\eps \to 0}\I_\eps \subset
\limsup\limits_{A\ni\eps \to 0}\I_\eps$,
and it is easy to see that the lower and upper limits
of $(\I_\eps)_{\eps\in A}$ are ideals in $\O(\Omega)$.
 If it is clear from
the context which set $A$ we are referring to, then we drop the 
symbol $A\ni$ from the subscript.  Properties following from this
definition are given in Section \ref{conv.of.id}.

This is equivalent to the notion of convergence inherited from the Douady space, see Proposition \ref{conv.in.dou} and the Remark following it, in
Section \ref{doucy}.

The inequality always between the limit Green function and
the Green function of a limit ideal survives under much weaker 
hypotheses.  

\begin{prop}
\label{gilelim}
If $V(\liminf\limits_{\eps \to 0} \mathcal I_\eps)=\{0\}$, then
$$
G_{\liminf\limits_{\eps \to 0} \mathcal I_\eps} \le \liminf_{\eps \to 0}G_{\mathcal I_\eps }
+O(1).
$$
\end{prop}
This Proposition will be proved in Section \ref{gineq}, after Lemma \ref{glb}.

Still without convergence of ideals, but under more stringent convergence
hypotheses for the Green functions, we have the following.

\begin{prop}
\label{uniflimsup}
If $G_{\mathcal I_\eps}$ converges uniformly
to $g$ on compact subsets of $\Omega \setminus \{0\}$, then
$G_{\limsup\limits_{\eps \to 0} \mathcal I_\eps} \le g$.
\end{prop}
This is proved in Section \ref{gineq}, after Lemma \ref{boundcv}.

\subsection{Main results}

In some privileged situations, equality will hold in Proposition
\ref{glimunif}. 

\begin{defn}
\label{defgci} 
The family of ideals $(\mathcal I_\eps)$ satisfies the \emph{Uniform Complete Intersection
Condition} if for any $\eps$, there exists a map $\Psi_0$ and maps $\Psi_\eps$ 
 from a neighborhood
of $\overline \Omega$ to $\C^n$ such that $\Psi_0$  is proper 
from $\Omega$ to $\Psi_0(\Omega)$, and
\begin{enumerate}
\item 
$\{a_j^\eps, 1 \le j \le N\} = \Psi_\eps^{-1}\{0\}$, for all $\eps$;
\item
For all $\eps \neq 0$, $1\le j \le N$ and $z$ in a neighborhood of $a_j^\eps$,
$$
\left| \log \|\Psi_\eps(z) \| - \log \|z - a_j^\eps \|
\right| \le C(\eps) < \infty ;
$$
\item
$\lim_{\eps\to 0}  \Psi_\eps = \Psi_0=( \Psi_0^1, \dots,  \Psi_0^n)$, uniformly on  $\overline \Omega$.
\end{enumerate}
\end{defn}
Notice that the first two conditions imply
$\mathcal I_\eps = \langle  \Psi_\eps^1, \dots,  \Psi_\eps^n \rangle.$

\begin{theorem}
\label{thmgci}
Let $(\I_\eps)$ be a family of ideals satifying the complete
intersection condition, set $S_\eps=V(\I_\eps)$ and
$\I_0=\langle \Psi^1_0,\dots,\Psi_0^n\rangle$.  Then 
\begin{enumerate}
\item
$\lim\limits_{\eps\to 0} \mathcal I_\eps = \mathcal I_0$,
\item
$\lim\limits_{\eps\to 0} G_\eps = G_{\mathcal I_0}$,
and the convergence is locally uniform on $\Omega \setminus \{0\}$.
\end{enumerate}
\end{theorem}

This theorem is proved in Section \ref{secproofs}.

We shall see that the above equalities only rarely hold. 
We need to recall the notions of length and multiplicity for an ideal.

\begin{defn}
\label{hsmult}
Let $\mathcal I$ be an ideal of $\mathcal O(\Omega)$ 
such that $V(\mathcal I)$ is a finite
set.
\begin{enumerate}
\item
The \emph{length} of $\I$ is 
$\ell (\mathcal I):= \dim \mathcal O/\mathcal I$. 
\item
The \emph{Hilbert-Samuel multiplicity} of $\mathcal I$ is 
$$ e(\mathcal I)=\lim_{k\to\infty}\frac{n!}{k^n}\,\ell(\mathcal I^k).$$
\item 
If $\mathcal I$ admits a set of $n$ generators, it is called
a parameter ideal, or \emph{complete
intersection} ideal.
\end{enumerate}
\end{defn}

Multiplicity is no smaller than length, and they match only in very
particular cases.

\begin{prop}{\cite[Ch. VIII, Theorem 23]{Za-Sa}}
\label{leqe}
If $V(\mathcal I) = \{a\}$, 
$e(\mathcal I) \ge \ell(\mathcal I)$, and 
$e(\mathcal I) = \ell(\mathcal I)$ if and 
only if $\mathcal I$ is a complete
intersection ideal.
\end{prop}

The next theorem is our main general result.
Its converse direction 
shows that the second conclusion of Theorem \ref{thmgci} 
can be deduced from hypotheses about the limit ideal only. 

\begin{theorem}
\label{thmiff}
Let  $\mathcal I_\eps= \mathcal I(S_\eps)$, 
where $S_\eps$ is a set of $N$ points all tending to $0$ and
assume that
$\lim_{\eps\to 0} \mathcal I_\eps = \mathcal I$.
Then $(G_{\mathcal I_\eps})$ converges to $G_{\mathcal I}$ locally
uniformly on $\Omega \setminus \{0\}$ if and only if 
$\mathcal I$ is a complete intersection ideal.
\end{theorem}

This theorem is proved in Section \ref{secproofs}.

\subsection{Examples}

Note that in the case where $\mathcal I$ requires more than $n$
generators, the sequence  $(G_{\mathcal I_\eps})$ may converge,
to a limit which is not the Green function of an ideal.  The
first case when this can occur is when $n=2$ and $N=3$.

We fix some notations. As usual, for $a \in \Omega$, 
$$\frak M_a := \mathcal I_{\{a\}} = \langle z_1-a_1, \dots, z_n-a_n \rangle. $$ 
Recall that 
$$\ell (\frak M_a) =1, \quad
\ell (\frak M_a^p ) = \left( 
\begin{array}{c} 
n+p-1\\ n 
\end{array} \right), p\ge 1.
$$
If $\mathcal I$ is the ideal of all functions vanishing at $N$ distinct
points $a_1, \dots, a_N$, then $\ell(\mathcal I) =N$. 

For $\eps \in \C$, let $a_i^\eps\in \Om \subset \C^2$,  $1\le i \le 3$, be three
distinct points,  
$S_\eps := \{ a_1^\eps, a_2^\eps, a_3^\eps \}$, with 
$\lim_{\eps\to0} a_i^\eps=0$,  $1\le i \le 3$;
$G_\eps:= G_{S_\eps}$, $\I_\eps:= \I(S_\eps)$.

Let $v_{ i}^\eps := [a_j^\eps-a_k^\eps] \in \P^1$, where $\{ i,j,k\}=\{1,2,3\}$, 
$j,k \neq i$.

\begin{theorem}
\label{gen3pt}
Assume that there exists $A\subset \C\setminus\{0\}$ 
with $0\in \bar A$ such that 
$v_{ i} = \lim_{A\ni\eps\to0} v_{ i}^\eps$. 
Let $\tilde v_{ i} \in \C^2$, $1\le i \le 3$,
be such that $\|\tilde v_{ i}\|=1$ and $[\tilde v_{ i}]= v_{ i}$.
\begin{enumerate}
\item[(i)]
If there exist $i, j$ such that 
 $v_{i} \neq v_{ j}$,  then $\lim_{\eps\to 0} \mathcal I_\eps = 
\frak M_0^2$, 
and 
$ G_{\eps}\to  g$ locally uniformly on $\Om\setminus \{0\}$
as $A\ni\eps\to 0$, where
$g \in PSH(\Om)$ is  extremal on $\Om\setminus \{0\}$, 
$g(z) \le \frac32 \log \|z\| + O(1)$ as $z\to 0$, 
and for $\zeta \in \C$ we have
$$
g(\zeta v) 
=\begin{cases} 2 \log |\zeta| + O(1), &v=\tilde v_{ i}, \ 1\le i \le 3,\\
\frac32 \log |\zeta| + O(1),  & v \in \C^2
\setminus \cup_{1\le i \le 3} \C \tilde v_{ i} .
\end{cases}
$$
\item[(ii)]
There exists an infinite set $\frak A$ of families 
$a:=\left\{\{a_1^\eps,a_2^\eps,a_3^\eps\}, \eps \in \C\right\}$, such that for
any $a \in \frak A$, 
$v_{1} =v_{ 2}=v_{ 3}=[1:0]$,  
$\lim_{\eps\to 0} \mathcal I_\eps = \I^a$ 
and $\lim_{\eps\to 0} G_\eps = g^a$,
where the ideals $\I^a$ (respectively the functions
$g^a$) are distinct for distinct values of $a$.
\end{enumerate}
\end{theorem}

This theorem is proved in Section \ref{ex}.

{\bf Remark 1.}
 In Section \ref{ex}, we also 
 give a general picture of the convergence (or not) of Green
functions for two poles tending to the origin, and
in the case of three poles,
after Proposition \ref{model},  
more detailed descriptions of the possible functions $g$, which
depend on whether the $v_i$ are all distinct or not.

{\bf Remark 2.}
Since $\P^1$ is compact, any family $\{v_i^\eps\}_{\eps \in E}$
admits a convergent subfamily. For a $E'\subset E$ such that each
$\{v_i^\eps\}_{\eps \in E'}$ is convergent, one of the above cases
will apply.  If different subfamilies lead to distinct limits, 
the original family $\{G_\eps, \eps \in E\}$ is not convergent.  In particular, if $\mathcal V$
denotes the cluster set of the sets $\{v_i^\eps, 1\le i \le 3\}$, 
and if we have $\# \mathcal V \ge 4$, the
family $\{G_\eps, \eps \in E\}$ is not convergent.

\subsection{Open questions}

When dealing with plurisubharmonic functions, it is natural to consider
``weak" convergence, in the $L^1_{\rm{loc}}$ sense. This is in fact enough to obtain first estimates of the mass of a limit of Green functions.

\begin{prop}
\label{massdecr}
Let $S_\eps$ be a set of $N$ points all tending to $0$.
Suppose that 
$g=\lim_{\eps\to0} G_{S_\eps}$ exists in $L^1_{loc}$. 
Then
$$
(dd^c g)^n (\Omega) \le \liminf_{\eps\to0} (dd^c G_{S_\eps})^n (\Omega) = N.
$$
\end{prop}

The proof of this useful fact, which relies mostly on a result in \cite{Ce3}, will be given in Section \ref{gineq}, before Lemma \ref{glb}.

However, to obtain a better behavior we would like to make sure, in particular, that the Monge-Amp\`ere mass of the limit function is concentrated at the origin.  This happens when convergence
is locally uniform on $\Omega \setminus \{0\}$ (Corollary
\ref{coridineq}).
Given the rigid nature of Green functions, is it possible to find any
situation where this better kind of convergence is {\it not} realized? 

\vskip.3cm

Theorem \ref{gen3pt} and the following remarks
 show that the same ideal $\frak M_0^2$
can be obtained as limit of many distinct families $\mathcal I_\eps$
of ideals based on three points, which give rise to distinct families of Green functions with different limits $g$, all of which, however, share the property that 
$(dd^c g)^2 = 3 \delta_0= \ell (\frak M_0^2) \delta_{V(\frak M_0^2)}$.

Is there any natural way to associate to an ideal $\mathcal I$
such that $V(\mathcal I)=\{0\}$ a plurisubharmonic function
$h_{\mathcal I}$ such that $(dd^c h_{\I})^n = \ell(\mathcal I)
\delta_0$ ? 

Any such correspondence, however, is not likely to be one-to-one ; 
for instance, if we take any two independent linear forms 
$\psi_1, \psi_2$ on $\C^2$, then in general $\langle \psi_1^2, \psi_2^2 \rangle
\neq \langle z_1^2, z_2^2 \rangle$; however they have the same Green functions,
which by Theorem \ref{thmgci} are the limits of Green functions of 
certain families of 4-point sets. 

\subsection{Acknowledgements \& origins of the question}

It is known that the Green function is smaller than another holomorphic invariant, the Lempert function \cite{Co}, \cite{CarlWieg}. 
When looking for cases where this inequality is strict,
for single poles, we had to consider situations where
$S=S_\eps$ depended on a parameter, and 
the poles of $S$ were tending to a same point as $\eps \to 0$ \cite{Th-Tr}.
This led us to consider $\lim_{\eps\to0} G_{S_\eps}$. 

The work on this paper spanned many years, from the first time 
that Nguyen Van Trao asked the last-named author about the limit
of a Green function with three poles, to the present. 
Over that period, we have benefitted from many conversations with kind
and patient colleagues about one aspect or another of this work. Although 
writing down a list runs the risk of omission, we'd like to thank (from A to Z) Eric Amar, Mats Andersson, Eric Bedford, Jean-Paul Calvi, 
Urban Cegrell, Jean-Pierre Demailly, Philippe Eyssidieux, Vincent Guedj, 
Nguyen Quang Dieu, Mikael Passare, Evgeny Poletsky, Mark Spivakovsky, Elizabeth Wulcan, Alain Yger, and Ahmed Zeriahi. 

\section{Convergence of ideals}
\label{conv.of.id}

\begin{lemma}
\label{lenglimsup} 
If $\ell(\mathcal I_{\eps})\le N$ for all $\eps \in A$, then 
$\ell (\limsup\limits_{A\ni\eps\to 0} \mathcal I_\eps) \le N$.
\end{lemma}

\begin{proof}
Suppose $k>N$ and $f^1, \dots, f^k \in \mathcal O (\Omega)$.
By the hypothesis, for any $\eps \in A$, there exist
$\lambda_\eps^m \in \C, 1 \le m \le k$, not all zero, such that
$$
\sum_{m=1}^k \lambda_\eps^m f^m = g_\eps \in \mathcal I_\eps.
$$
We can normalize the coefficients $\lambda_\eps^m $ such that
$\max_m |\lambda_\eps^m |=1$, and by passing to a subsequence,
and renumbering the $f^m$ if needed, 
we can assume $\lambda_\eps^1 =1$, $|\lambda_\eps^m |\le 1$ for
$2\le m \le k$.  Then by using the compactness of $\overline \D$,
we can pass to another subsequence $\eps_j$ such that
$\lim_{j\to\infty} \lambda_{\eps_j}^m =  \lambda^m \in \overline \D$,
$2\le m \le k$.  This implies that $\lim g_{\eps_j}=:g$ exists 
(locally uniformly) and
$$
f^1+\sum_{m=2}^k \lambda^m f^m = g \in \limsup\limits_{A\ni\eps\to 0} 
\mathcal I_\eps.
$$
So the system $[f^1], \dots, [f^k]$ is not linearly independent in
$\mathcal O (\Omega) / \limsup\limits_{A\ni\eps\to 0} \mathcal I_\eps $.
\end{proof}

\begin{lemma}
\label{lengliminf}
If for each $\eps \in A$, $\mathcal I_{\eps}$ is an ideal based on $N$ distinct points, all tending to $0$ as $\eps$ tends to $0$, and if 
$\ell(\mathcal I_{\eps})\ge N$ for all $\eps \in A$, then 
$\ell (\liminf\limits_{A\ni\eps\to 0} \mathcal I_\eps) \ge N$.
\end{lemma}

\begin{proof}
Let $\pi_j$ denote the projection to the $j$-th coordinate axis.
Let $V(\mathcal I_\eps)= \{a_1^\eps, \dots, a_N^\eps\}$.

We can split $A$ into a finite union of sets $A_k$ such that 
for each $k$, $j \in \{1,\dots,n\}$, 
$\#\pi_j(\{a_1^\eps, \dots, a_N^\eps\})=N_{k,j}$ is independent of $\eps$.
The sets $A_k$ which don't have the origin in their closure do not
concern us; let us now consider one that does, which we will denote again
by $A$, and write $N_j=N_{k,j}$.  Then let 
$$
P_\eps := \pi_1(\{a_1^\eps, \dots, a_N^\eps\}) \times \cdots
\times \pi_n(\{a_1^\eps, \dots, a_N^\eps\})  \mbox{ (cartesian product)}
$$
and $\mathcal J_\eps$ the set of all $f\in \O(\Omega)$ vanishing
on $P_\eps$.  It is easy to see that $\ell (\mathcal J_\eps)
= \#P_\eps = \prod_{j=1}^n N_j \le N^n$.

\begin{lemma}
\label{prodlem}
Under the hypotheses of Lemma \ref{lengliminf},
$$
\lim_{\eps\to0} \mathcal J_\eps = \mathcal J
:= \left\{ f \in \O(\Omega) : \frac{\partial^{k_1+\dots+k_n} f}
{\partial z_1^{k_1} \dots \partial z_n^{k_n}}(0)=0,
1 \le j \le n, 0\le k_j \le N_j-1
\right\}. 
$$
\end{lemma}

Lemma \ref{prodlem} will be proved below. Note that it is a special
case of the first statement in Theorem \ref{thmgci}.

Let $\mathcal I := \liminf\limits_{A\ni\eps\to 0} \mathcal I_\eps$. 
Since $\mathcal J_\eps \subset \mathcal I_\eps$, 
it is easy to see that $\mathcal J\subset \mathcal I$.

Then
$\ell (\mathcal I) = \ell (\mathcal J) - \dim (\mathcal I/\mathcal J)
= \prod_{j=1}^n N_j - \dim (\mathcal I/\mathcal J)$, and
$$N\le \ell (\mathcal I_\eps) = \ell (\mathcal J_\eps) - \dim (\mathcal I_\eps/\mathcal J_\eps) = \prod_{j=1}^n N_j - \dim (\mathcal I_\eps/\mathcal J_\eps). 
$$  
So to prove Lemma \ref{lengliminf}, it will be enough to show that 
$\dim (\mathcal I/\mathcal J) \le \prod_{j=1}^n N_j -N$,
by using the fact that
$\dim (\mathcal I_\eps/\mathcal J_\eps) \le \prod_{j=1}^n N_j -N$.

Let $k > \prod_{j=1}^n N_j -N$ and $f^1, \dots, f^k \in \mathcal I$.
By the definition of $\liminf$, for any $\eps \in A$, there exist
$f^1_\eps, \dots, f^k_\eps \in \mathcal I_\eps$ such that
$\lim_{\eps\to0} f^m_\eps = f^m$, $1\le m \le k$. By the bound on
the dimension of $\mathcal I_\eps/\mathcal J_\eps$, there exists
$g_\eps \in \mathcal J_\eps$ and coefficients $\lambda_\eps^m \in \C, 1 \le m \le k$, not all zero, such that
$$
\sum_{m=1}^k \lambda_\eps^m f^m_\eps = g_\eps .
$$
Now we pass to a subsequence exactly as in the proof of Lemma 
\ref{lenglimsup}, using in addition the convergence of the $f^m_\eps $,
to find 
$$
f^1 + \sum_{m=2}^k \lambda^m f^m = g ,
$$
with $g = \lim g_{\eps_j} \in \limsup\limits_{A\ni\eps\to 0} \mathcal J_\eps = \mathcal J$, since the family $(\mathcal J_\eps )$ converges. This
proves that the system $[f^1], \dots, [f^k]$ is not linearly independent in
$\mathcal I/\mathcal J$.
\end{proof}

\begin{cor}
\label{lenglim}
If $(\mathcal I_\eps)$ is a convergent family of
point-based ideals,
$\ell (\lim_{\eps \to 0}\mathcal I_\eps) =
\lim_{\eps \to 0} \ell (\mathcal I_\eps)$.
\end{cor}

\begin{proof}[Proof of Lemma \ref{prodlem}]

Denote the elements of $\pi_j(\{a_1^\eps, \dots, a_N^\eps\}) $
by $\alpha^{i,\eps}_j$, $1\le i \le N_j$, $1\le j \le n$. Let 
$\psi^\eps_j(\zeta):= (\zeta - \alpha^{1,\eps}_j) \cdots (\zeta - \alpha^{N_j,\eps}_j)$, for $\zeta \in \C$. 

Suppose that  $f\in \mathcal J$. Then there are holomorphic functions
$h_j$, $1\le j \le n$, such that 
$$f(z) = \sum_{j=1}^n z_j^{N_j} h_j(z).
$$
Setting 
$$
f_\eps (z) = \sum_{j=1}^n \psi^\eps_j(z_j) h_j(z),
$$
we have a family of $f_\eps \in \mathcal J_\eps$ such that
$f=\lim_{\eps \to 0} f_\eps$.

Conversely, let $\lim_{k\to \infty} f_{\eps_k} = f$, where 
$f_{\eps_k} \in \mathcal J_{\eps_k}$, ${\eps_k}\to 0$. 
By rescaling, we might assume that $\overline \D^n \subset \Omega$.
One can prove by induction on $n$ that 
if $f_\eps \in \mathcal J_\eps$, and $|\eps|$ is small
enough so that $|\alpha^{i,\eps}_j|<1$ for all $i, j$, then 
$$
\int_{(\partial \D)^n} \frac{f_\eps(z_1, \dots, z_n)}
{\prod_{j=1}^n\psi^\eps_j(z_j)} dz_1 \dots dz_n = 0.
$$
Applying this to $f_{\eps_k}$,
and passing to the limit as $k\to \infty$, we find 
$$
\frac{\partial^{N_1+\dots+N_n-n} f}{\partial z_1^{N_1-1} \dots 
\partial z_n^{N_n-1}}(0)= \frac1{(2\pi i)^n}
\int_{(\partial \D)^n} \frac{f(z_1, \dots, z_n)}{\prod_{j=1}^n
  z_j^{N_j} }dz_1 \dots dz_n = 0. 
$$
The same result will hold replacing each $N_j$ by any $k_j\le N_j$,
simply by taking appropriate subsets of $P^\eps$.
\end{proof}

\section{The Douady space and the cycle space of \ $\Omega$}
\label{doucy}

The aim of this section is to show that the convergence notion we introduced in the previous section is in fact  equivalent to convergence in the Douady space of \ $\Omega$ (see proposition \ref{conv.in.dou} below.)

\subsection{The Douady space of \ $\Omega$}

A \emph{flat and proper family of subspaces} of \ $\Omega$ \ is a pair of complex spaces \ $(S,Z)$ \ such that \ $Z$ \ is a subspace of \ $S\times \Omega$ \ and such that the natural projection \ $\pi : Z\rightarrow S$ \ is a flat and proper holomorphic map. The space \ $S$ \ is called the \emph{parameter space} and the  space \ $Z$ \ is called the \emph{graph} of the family.

\smallskip
Due to the natural identification \ $\{s\}\times \Omega\simeq \Omega$ \ we think of this as a family of compact subspaces \ $(Z_s)_{s\in S}$ \ in \ $\Omega$ \ parametrized by the space \ $S$ \ where \ $Z_s$ \ denotes the analytic fibre of \ $\pi$ \ over the point \ $s$. 

\smallskip
Since \ $\pi$ \ is a proper holomorphic map the spaces \ $Z_s$ \ are all {\em compact} complex subspaces of \ $\Omega$ \ and consequently finite.  Hence \ $\pi$ \ is a finite (and proper) holomorphic map and in that case one knows that \ $\pi$ \ is flat if and only if the direct image sheaf \ $\pi_*\mathcal{O}_Z$ \ is a locally free \ $\mathcal{O}_S-$module (of finite rank); see 
\cite{Dou2}. 
     This again is equivalent to say that, on each connected component of \ $S$, all the fibres of the sheaf \ $\pi_*\mathcal{O}_Z$ \ are complex vector spaces of the same dimension.  But the fibre of \ $\pi_*\mathcal{O}_Z$ \ at the point \ $s$ \ is naturally isomorphic to the quotient space \ $\Gamma(\Omega,\mathcal{O}_{\Omega})/I_s$ \ where \ $I_s$ \ denotes the ideal of   \ $Z_s$ \ in \ $\Gamma(\Omega,\mathcal{O}_{\Omega})$. In other words the map \ $\pi$ \ is flat if and only if the length of \ $I_s$ \ is a locally constant function on \ $S$.

\smallskip
From 
\cite{Dou1} 
     we know  that every complex space admits a \emph{universal flat and proper family}. In our setting this means that there exists a flat and proper family \ $(D,X)$ \ of subspaces in \ $\Omega$ \ having the following universal property:
\begin{itemize}
\item
If \ $(S,Z)$ \ is any flat and proper family of subspaces of \ $\Omega$ \  then there exists a {\em unique} holomorphic map \ $f : S\rightarrow D$ \ such that \ $Z$ \ is the pullback of \ $X$ \ by \ $f\times\id_{\Omega}$.
\end{itemize}
 
This implies in particular that the space \ $D$, which is by definition the {\em Douady space} of \ $\Omega$, parametrizes (in a one-to-one way) all the compact subspaces of \ $\Omega$; in other words every compact subspace of \ $\Omega$ \ appears exactly once in the family \ $(X_t)_{t\in D}$.  In the sequel we will denote \ $\mathcal{I}_t$ \ the \ $\mathcal{O}_{\Omega}$-ideal corresponding to \ $X_t$ \ and put \ $I_t := \Gamma(\Omega,\mathcal{I}_t)$.

\smallskip
We have a natural decomposition \ $D = \sqcup_{k\geq 1}D_k$, where \ $D_k$ denotes the open subspace of \ $D$ \ formed by those \ $t$ \ such that \ $I_t$ \ is of length \ $k$.

\subsection{The cycle space of \ $\Omega$}

Let \ $\Sym^k(\Omega)$ \ denote the \ $k$-th symmetric product of \ $\Omega$, i.e. the normal complex space obtained as a quotient of \ $\Omega^k$ \ by the natural action of the \ $k$-th symmetric group.   One can think of every element in \ $\Sym^k$ \ as a given set of points each with a multiplicity. These elements are usually called \ \emph{$0$-cycles} and  each one of them can be expressed in a unique way as \ $n_1x_1 + \cdots + n_lx_l$, where \ $x_1,\ldots,x_l$ \ are mutually distinct, $n_j$ \ is the multiplicity of \ $x_j$ \ and consequently \ $n_1+\cdots +n_l = k$.

Since every compact complex subspace of \ $\Omega$ \ is finite  the disjoint union 
$$
\mathcal{C} := \bigsqcup_{k\geq 1}\Sym^k(\Omega)
$$
is the \emph{cycle space} of \ $\Omega$.

\medskip
For every \ $k$ \ we have a natural holomorphic map from the reduction of \ $D_k$ \ to \ $\Sym^k(\Omega)$ \ defined in the following way (see for instance \cite{Ba} or \cite{Ma}): 

\smallskip
To each \ $t$ \ in \ $D_k$ \ we associate the \ $0$-cycle \ $n_1x_1 + \cdots + n_lx_l$ \ where \ $x_1,\ldots,x_l$ \ are the mutually distinct points of \ $X_t$ \ and \ $n_j := \dim_{\C}\mathcal{O}_{\Omega,x_j}/(\mathcal{I}_t)_{x_j}$. 

\smallskip
Moreover this map is proper  \cite{Na}. Hence  we obtain a proper holomorphic map 
$$
\mu : D_{red}\longrightarrow \mathcal{C}
$$
where \ $D_{red}$ \ denotes the reduction of \ $D$.

\subsection{Topology of the Douady-space of \ $\Omega$}

Since the topological space underlying \ $D$ \ is first countable the following proposition characterizes its topology.

\begin{prop}\label{conv.in.dou}
Let \ $(t_j)_{j\geq 1}$ \ be a sequence in \ $D$ \ and let  \ $a$ \ be a point in \ $D$. Then the sequence \ $(t_j)_{j\geq 1}$ \ converges to  \ $a$ \ in \ $D$ \ if and only if the following two conditions are satisfied:
\begin{enumerate}
\item
\ The sequence  \ $(\mu(t_j))_{j\geq 1}$ \ converges to  \ $\mu(a)$ \ in \ $\mathcal{C}$.
\item 
\ $\lim\limits_{j\to \infty}I_{t_j} = I_a$, where the limit is taken in the sense of Definition \ref{convid}.
\end{enumerate}
\end{prop}

\begin{proof} 
Suppose that the sequence \ $(t_j)_{j\geq 1}$ \ converges to \ $a$ \ in \ $D$. Then condition $(1)$ is satisfied because the map \ $\mu$ \ is continuous. \\
To prove that condition $(2)$ is fulfilled it is sufficient to show that \ $\lim_{t\to a}I_t = I_a$. Without loss of generality we may replace \ $D$ \ by a Stein open (connected) neighbourhood \ $T$ \ of \ $a$ \ in \ $D$ \ and we may assume \ $T$ \ is reduced since flatness is preserved by base change.  We still denote \ $X$ \ the restriction of the graph \ $X$ \ to \ $T$ \ and let \ $\mathcal{I}$ \ be the corresponding \ $\mathcal{O}_{T\times\Omega}$-ideal.\\
Let us first prove that \ $I_{a}\subseteq\liminf\limits_{t\to a}I_t$.  To do so take any function \ $f$ \ in \ $I_{a}$. Then since \ $T\times\Omega$ \ is Stein there exists a function \ $g$ \ in  \ $\Gamma(T\times\Omega,\mathcal{I})$ \ such that \ $g(a,z) = f(z)$ \ for all \ $z$ \ in \ $X_{a}$. Now for every \ $t$ \ in \ $T$ \ define a function \ $f_t$ \ on \ $\Omega$ \ by setting 
$$
f_t(z) := g(t,z).
$$
Then it is clear that \ $f_t\in I_t$ \ and we obviously have \ $f_t\to f$ \ uniformly on every compact set in \ $\Omega$.\\
To prove that \ $\limsup\limits_{t\to a}I_t\subseteq I_{a}$ \ suppose we have a sequence of points \ $(x_{\nu})_{\nu}$ \ in \ $T$ \ converging to \ $a$ \ and  for each \ $\nu$ \ a function \ $f_{\nu}$ \ in \ $I_{x_{\nu}}$ \ such that  the sequence  \ $(f_{\nu})_{\nu}$ \ converges locally uniformly to a function \ $f$ \  in \ $\Gamma(\Omega,\mathcal{O}_{\Omega})$. Now each \ $f_{\nu}$ \ defines a global holomorphic section \ $\sigma_{\nu}$ \ of the locally free sheaf \ $\pi_*\mathcal{O}_X$ \ on \ $T$ \ and they converge locally uniformly to the holomorphic global section \ $\sigma$ \ defined by \ $f$ \ as \ $\nu\to \infty$. Now if \ $\tau$ \ is a global section of \ $\pi_*\mathcal{O}_X$ \ defined by a global holomorphic function \ $g$ \ on \ $\Omega$ \ then the \lq\lq value\rq\rq,  \ $\tau(t)$,   of  \ $\tau$ \ at a point \ $t$ \ in \ $T$ \ is the image of \ $g$ \ in  the \ $\C-$vector space
$$
(\pi_*\mathcal{O}_X)_t\otimes_{\mathcal{O}_{T,t}}\C \ \simeq \ \Gamma(X_t,\mathcal{O}_{X_t}) \ \simeq \ \Gamma(\Omega,\mathcal{O}_{\Omega})/I_t
$$
In particular  \ $\tau(t) = 0$ \ is equivalent to \ $g\in I_t$. Since \ $\sigma_{\nu}(t_{\nu}) = 0$ \ it follows that \ $\sigma(a) = 0$ \ and consequently \ $f\in I_{a}$. 

\medskip
Conversely, suppose that the sequence \ $(t_j)_{j\geq 1}$ \  satisfies the two conditions and assume that it does not tend to \ $a$.  Then the point \ $a$ \ admits an open neighbourhood \ $V$ \ in \ $D$ \ outside of which there exists a subsequence  \ $(t_{j_l})_{l\geq 1}$ \ of  \ $(t_j)_{j\geq 1}$.

Since the map \ $\mu : D_{red}\longrightarrow \mathcal{C}$ \ is proper we may, without loss of generality, assume that the subsequence  \ $(t_{j_l})_{l\geq 1}$ \  converges to a point \ $b$ \  in the fibre \ $\mu^{-1}(\mu(a))$. From what we proved above it then follows that
$$
I_a = \lim_{j\to\infty}I_{t_j} = \lim_{l\to\infty}I_{t_{j_l}} = I_b
$$
in contradiction to the fact that \ $t_{j_l}\notin V$ \ for all \ $l$.
\end{proof}

{\bf Remark.} Assuming the hypothesis of Section \ref{conv.of.id} let us denote \ $\iota : A\rightarrow D$ \ the canonical mapping that associates to each \ $\epsilon$ \ in \ $A$ \ the compact complex subspace of \ $\Omega$ \ defined by \ $I_{\epsilon}$. If \ $I_0 := \lim_{\epsilon\to 0}I_{\epsilon}$ \ exists then \ $I_0$ \ defines a point in \ $D$ \ and we get
$$
\lim_{\epsilon\to 0}\iota(\epsilon) = \iota(0).
$$

\section{General inequalities}
\label{gineq}

Suppose throughout this section that we are given a family of ideals 
$\mathcal I_\eps \subset \mathcal O(\Omega)$. 

\subsection{First estimates}
\label{firstest}
\begin{lemma}
\label{roughest}
Let $\Omega$ be a bounded hyperconvex domain such that 
$0\in \Omega$. Let $S_\eps = \{a_j^\eps, 1\le j \le N\}$,
with $\lim_{\eps\to0} a_j^\eps= 0, 1\le j \le N$. 

Then for any $\delta>0$, there exists $\eps_0=\eps_0(\delta)>0$
such that for any $z \in \Omega \setminus \bar B (0,\delta)$,
for any $\eps$ such that $ |\eps| \le \eps_0$,
$$
(N+\delta) G_{0}(z) \le G_{S_\eps} (z) \le (1-\delta) G_{0}(z).
$$
\end{lemma}

{\bf  Remark.} This implies that the family $G_{S_\eps}$
is equicontinuous near $\partial \Omega$. As a consequence, 
a subsequence $G_{S_{\eps_j}}$ converges uniformly on compacta
of $\overline \Omega \setminus \{0\}$ if and only if it 
converges uniformly on compacta
of $ \Omega \setminus \{0\}$.

\begin{proof}
It is well known that 
$$
\sum_{j=1}^N G_{a_j^\eps}(z) \le G_{S_\eps} (z) \le \min_{1\le j \le N} G_{a_j^\eps}(z).
$$
We will compare each of the $G_{a_j^\eps}(z) $ to $ G_{0}(z) $. 
There are all equal to $0$ on $\partial \Omega$. 

There are numbers $0<r_1<r_2$, $r_2 \ge 1$, such that $\bar B(0,r_1) \subset \Omega \subset B(0,r_2)$,
so 
$$
\log \frac{\|z\|}{r_2} \le G_{0}(z) \le \log \frac{\|z\|}{r_1} .
$$

Now we take $z$ such that $\|z\|=\delta_1 \le \delta$, $\delta_1$ to be chosen
below. Then for $|a_j^\eps| < \delta$, 
$$
G_{a_j^\eps}(z) \le \log\frac{\|z-a_j^\eps\|}{r_1 - \|a_j^\eps\|} 
\le \log \frac{\delta_1}{r_1} + 1,
$$
for $|\eps|$ small enough (depending on $\delta_1$).
We can choose $\delta_1$ so small that 
$$
 \log {\delta_1} - \log {r_1} + 1 \le
(1-\delta) \log {\delta_1} - \log {r_2}  \le (1-\delta)\log \frac{\delta_1}{r_2},
$$
so that $G_{a_j^\eps}(z) \le (1-\delta) G_{0}(z)$ when $\|z\|=\delta_1$.
Since $ G_{0}$ is maximal plurisubharmonic on $\Omega \setminus \bar B(0,\delta_1)$,
and $G_{a_j^\eps}(z)= (1-\delta) G_{0}(z)=0$ when $z\in \partial \Omega$,
the inequality holds on the whole of $\Omega \setminus \bar B(0,\delta_1)$.

In a similar way, for$\|z\|=\delta_1 \le \delta$, and $|a_j^\eps| < \delta$, 
$$
G_{a_j^\eps}(z) \ge \log\frac{\|z-a_j^\eps\|}{r_2 +\|a_j^\eps\|} 
\ge \log \frac{\delta_1}{r_2} - 1,
$$
for $|\eps|$ small enough.
We can choose $\delta_1$ so small that 
$$
 \log {\delta_1} - \log {r_2} -1 \ge
(1+\delta/N)\log \frac{\delta_1}{r_1}, 
$$
so that $G_{a_j^\eps}(z) \ge (1+\delta/N) G_{0}(z)$ when $\|z\|=\delta_1$,
and the inequality on the whole of $\Omega \setminus \bar B(0,\delta_1)$ follows
by maximality of $G_{a_j^\eps}$. 
\end{proof}

\begin{proof*}{\it Proof of Proposition \ref{massdecr}.}
The first claim in \cite[Lemma 2.1]{Ce3} states that if 
$u$, $u_j \in \mathcal F(\Omega)$ and $u_j$ converges weakly to $u$,
then for any $w\in PSH_-(\Omega)$,
$$
\limsup_{j\to\infty} \int_\Omega w (dd^c u_j)^n \le \int_\Omega w (dd^c u)^n.
$$
Since $\Omega$ is hyperconvex, for any compact $K\subset \Omega$ we can find a function
$w\in PSH_-(\Omega)$ such that $-1\le w$ on $\Omega$ and $w\equiv -1$ on $K$. 
If we can apply the above inequality to $u_j=G_{S_{\eps_j}}$ and $u=g$, we find the
desired inequality.

It remains to see that all those functions belong to Cegrell's class $\mathcal F(\Omega)$,
as defined in \cite[Definition 4.6]{Ce2}, or in \cite{Ce3}. Recall that $\mathcal E_0$
is the class of all bounded functions $v\in PSH(\Omega)$ which tend to $0$ at the
boundary and such that $\int (dd^c v)^n <\infty$. Then $u\in\mathcal F(\Omega)$
if and only if there exists a sequence $u_j$ decreasing to $u$ with $u_j\in \mathcal E_0$
and $\sup_j \int (dd^c u_j)^n <\infty$. 

By Lemma \ref{roughest}, for $j$ large enough, $G_{S_{\eps_j}}
\ge (N+\delta) G_{0}(z)$, and therefore their weak limit $g$ must satisfy the
same inequality (almost everywhere). Hence, for any natural number $m$,
$\max (-m ,g)$ and $\max (-m ,G_{S_{\eps_j}})$ 
have bounded Monge Amp\`ere mass, with the same bound, and vanish on the boundary of $\Omega$.
Since $g$ (resp. $G_{S_{\eps_j}}$) is the decreasing limit of $\max (-m ,g)$
(resp. $\max (-m ,G_{S_{\eps_j}})$), those functions do belong to the class $\mathcal F(\Omega)$.
\end{proof*}

\subsection{Inequalities involving limit ideals}

Let ${\mathcal I_*}:= \liminf_{\eps \to 0} \mathcal I_\eps$.
We make the additional assumption that $V({\mathcal I_*}) = \{0\}$. A general lower bound can then be expressed in terms of 
${\mathcal I_*}$. 

\begin{lemma}
\label{glb}
For any $\eta >0$ and any neighborhood $\omega$ of $0$, there exists
$\eps_0 >0$ such that for $|\eps|<\eps_0$,
$$
G_{\mathcal I_\eps } (z) \ge (e({\mathcal I_*})+\eta) G_{0} (z),
\mbox{ for } z\in \Omega \setminus \omega.
$$
In particular, $\liminf_{\eps \to 0}G_{\mathcal I_\eps } \ge e({\mathcal I_*}) G_{0}$.
\end{lemma}

In order to prove this Lemma, we need to use the notion
of integral closure of an ideal. 
\begin{defn}
\label{intclos}
The \emph{integral closure} $\bar{\mathcal I}$ 
of an ideal $\mathcal I$ of a ring $\mathcal A$
is the set of all $f \in \mathcal A$ such that there exist $m \in \N^*$
and $a_i \in \mathcal I^{m-i}$, $0 \le i \le m-1$ such that
$$
f^m + \sum_{i=0}^{m-1} a_i f^i = 0.
$$
We say that an ideal $\mathcal J\subset \mathcal I$ 
is a \emph{reduction} of
$\mathcal I$ if and only if $\bar{\mathcal J}=\bar{\mathcal I}$.
\end{defn}

It follows from the Brian\c con-Skoda Theorem \cite{Br-Sk} that 
$$
\bar{\mathcal I} = \left\{ u \in \mathcal O(\Omega) : |u| \le C 
 \max_i  |\psi_{a,i}|, \mbox{ near each } a \in \Omega
\right\},
$$
see e.g. \cite[Corollary 10.5]{De2}.

Thus 
$G_{\bar{\mathcal I}}^\Omega = G_{\mathcal I}^\Omega$. This provides
many examples of distinct ideals with the same Green function, and
is at the root of the phenomena of non-convergence.

\begin{proof*}{\it Proof of Lemma \ref{glb}.}

There exists a reduction $\mathcal J$ of ${\mathcal I_*}$ generated
by exactly $n$ functions, say $f^1, \dots, f^n\in \mathcal O(\Omega)$
(see e.g. \cite[Chapter VIII, Lemma 10.3, p. 394]{De2}). 
Let $f:= (f^1, \dots, f^n)$.
Since $G_{\mathcal J}= G_{\bar {\mathcal J}}= G_{{\mathcal I_*}}$, 
\cite[Theorem 2.5]{Ra-Si} implies $G_{{\mathcal I_*}}=
\log \|f\| +O(1)$, and the multiplicity of the mapping $f$ at $0$
equals $e({\mathcal I_*})$. 

By definition of $\liminf_{\eps \to 0} \mathcal I_\eps$, there are 
functions $f_\eps^j$ tending to $f^j$ locally uniformly on 
$\Omega$, for $1\le j \le n$. Then let
$$
\mathcal J_\eps := \langle f_\eps^1, \dots , f_\eps^n \rangle 
\subset \mathcal I_\eps.
$$
Let $m_{\eps,k}$ be the multiplicity at the point 
$a_{\eps,k}\in V(\mathcal J_\eps)$ of the mapping $f_\eps$. Therefore,
using for instance \cite[Chap. 2, Theorem 1, p. 60]{DA} (statement 8),
$$
\mathcal J_\eps \supset \cap_k \frak M_{a_{\eps,k}}^{m_{\eps,k}}
=: \mathcal K_\eps,
$$
which implies 
$$
G_{\mathcal I_\eps } \ge G_{\mathcal J_\eps } \ge
G_{\mathcal K_\eps } \ge \sum_k m_{\eps,k} G_{a_{\eps,k}}.
$$
By 
Rouch\'e's theorem (see e.g. \cite[\S 5.2]{Ts}), when $|\eps|<\eps_0$, 
$\sum_k m_{\eps,k} \le e({\mathcal I_*})$. Since each
$G_{a_{\eps,k}}\to G_{0}$ uniformly on $\Omega \setminus \omega$,
we have 
$$
G_{a_{\eps,k}} \ge \left( 1 + \frac\eta{e({\mathcal I_*})}
\right) G_{0}
$$
on  $\Omega \setminus \omega$ for $|\eps|<\eps_0$ and so,
$$
G_{\mathcal I_\eps } \ge
\sum_k m_{\eps,k} \left( 1 + \frac\eta{e({\mathcal I_*})}
\right) G_{0} \ge (e({\mathcal I_*}) +\eta) G_{0}.
$$
\end{proof*}

\begin{proof*}{\it Proof of Proposition \ref{gilelim}.}

As in the proof of Lemma \ref{glb} above, if 
$\mathcal J$ is a reduction of
${\mathcal I_*}$ and $\mathcal J =\langle f^1, \dots, f^n\rangle$,
\begin{equation}
\label{glogf}
G_{{\mathcal I_*}}=\log \|f\| +O(1). 
\end{equation}
Given some 
$\Omega'$ relatively compact in $\Omega$, we can assume 
$\sup_{\Omega'}\|f\| <1$. Then for $|\eps|<\eps_0$, 
$\sup_{\Omega'}\|f_\eps\| < 1$, where $f_\eps = (f_\eps^1, \dots ,
f_\eps^n)$, again as above. Thus for $z \in \Omega'$,  
\begin{equation}
\label{logfgp}
\log\|f_\eps(z)\| \le G_{\mathcal I_\eps}^{\Omega'} (z).
\end{equation}
By Lemma \ref{glb}, for all $z \in \partial \Omega'$, for $|\eps|<\eps_0$, 
$$
G_{\mathcal I_\eps}(z) \ge (e({\mathcal I_*}) +1) 
\inf_{\Omega \setminus \Omega'} G_{0}(z)=: A.
$$
Since $G_{\mathcal I_\eps}-A$ is maximal on $\Omega \setminus \{0\}$,
nonnegative on $\partial \Omega'$ and $G_{\mathcal I_\eps}-A= 
G_{\mathcal I_\eps}^{\Omega'}+O(1)$ near $0$, we have 
$$
G_{\mathcal I_\eps}-A \ge
G_{\mathcal I_\eps}^{\Omega'}
$$
on $\Omega'$, so \eqref{logfgp} gives, for $z \in \Omega'$, 
$$
\log\|f_\eps(z)\| +A \le G_{\mathcal I_\eps}(z).
$$
Since $f_\eps$ converges to $f$
uniformly on $\Omega'$, $\log\|f(z)\| +A \le 
\liminf_{\eps \to 0} G_{\mathcal I_\eps}(z)$ on $\Omega'$, so 
\eqref{glogf} implies $G_{\mathcal I_*} \le 
\liminf_{\eps \to 0} G_{\mathcal I_\eps}(z) +O(1)$.
\end{proof*}

{\bf  Remark.}  If $G_{\mathcal I_\eps}$ converges uniformly
to $g$ on compact subsets of $\Omega \setminus \{0\}$, then
Proposition \ref{gilelim} implies that $G_{{\mathcal I_*}} \le g$.

\subsection{Uniform convergence}

We start with a sufficient condition for uniform convergence
that will be useful in particular in Section \ref{ex}.
We need a bit of shorthand from \cite{Ce-Po}.

\begin{defn}
\label{defequiv}
Let $u_1, u_2 \in PSH(\Omega)$, such that $u_1^{-1}\{-\infty\} = u_2^{-1}\{-\infty\} = \{0\}$.
We say that $u_1$ and $u_2$ are \emph{equivalent} near $0$, and we write $u_1 \sim_0 u_2$, if
and only if there exists a neighborhood $U$ of $0$ such that $u_1-u_2|_U \in L^\infty(U)$.
\end{defn}

This implies that $(dd^c u_1)^n(\{0\}) = (dd^c u_2)^n(\{0\}) $.

\begin{lemma}
\label{boundcv}
Suppose that there exists a function $G$ from $\Omega$ to $[-\infty,0]$
and a constant $C>0$ such that for any $\delta \in (0,\delta_0]$, there exists $\eps(\delta)>0$
such that for any $\eps$ with $|\eps|<\eps(\delta)$, for any $z$ such that
$\|z\|=\delta$, 
$$
\left| G_{S_{\eps}} (z) - G(z)\right| \le C.
$$
Then $\lim_{\eps\to0} G_{S_{\eps}} (z)  = g(z)$, uniformly on compacta of
$\Omega \setminus \{0\}$, and clearly $g \sim_0 G$.
\end{lemma}

\begin{proof}
It is enough to see that for any $\delta_0>0$, for any $\eta >0$, 
there exists  $\eps_0>0$
such that for any $\eps_1, \eps_2$ with $|\eps_1|, |\eps_2|<\eps_0$, for 
any $z \in \Omega \setminus B(0,\delta_0)$, 
\begin{equation}
\label{cauchyseq}
(1+\eta) G_{S_{\eps_1}} (z) \le G_{S_{\eps_2}} (z) \le (1-\eta) G_{S_{\eps_1}} (z).
\end{equation}
By the hypothesis, and by Lemma \ref{roughest}, $G(z) \le \frac12 G_{0}(z)+C $. 
Therefore there exists $\delta_1\in (0, \delta_0)$ such that for $\|z\|=\delta<\delta_1$,
$$
(1+\eta) G(z) \le  G(z) - C < G(z) + C  \le (1-\eta) G(z).
$$
Using the hypothesis once again, there exists $\eps(\delta)>0$
such that for any $\eps$ with $|\eps|<\eps(\delta)$, for any $z$ such that
$\|z\|=\delta$, \eqref{cauchyseq} holds.
Both functions $G_{S_{\eps_1}}$ and $G_{S_{\eps_2}}$ are maximal
plurisubharmonic, and equal to $0$ on $\partial \Omega$, so those
inequalities extend to $\Omega \setminus B(0,\delta) \supset \Omega \setminus B(0,\delta_0)$.
\end{proof}

\begin{proof*}{\it Proof of Propositions \ref{glimunif} and \ref{uniflimsup}.}

Denote 
$\mathcal I^*:= \limsup_{\eps \to 0} \mathcal I_\eps$.
Let $h\in \mathcal I^*$, $\sup_{\Omega'}|h| <1$; let $(h_{\eps_j})_j$
be a sequence of holomorphic functions converging uniformly to $h$
 such that $h_{\eps_j} \in \mathcal I_{\eps_j}$. Then 
 $\log |h_{\eps_j}| \le G_{\mathcal I_{\eps_j}}^{\Omega'}$ on
 $\Omega'$ and so, as in the proof of Proposition \ref{gilelim},
 $$
 \log |h_{\eps_j}| \le G_{\mathcal I_{\eps_j}}^{\Omega} +A.
 $$
 Therefore $\log|h| \le g+A$ and thus, applying this to any generator
 of $\mathcal I^*$, $G_{\mathcal I^*} \le g$. 
 
 Proposition \ref{glimunif} then follows as a special case. 
\end{proof*}

\begin{proof*}{\it Proof of Corollary \ref{coridineq}.}
Let
$ g:= \lim_{\eps\to0} G_{\mathcal I_\eps}$,
$ \mathcal I:=\lim_{\eps\to0}  \mathcal I_\eps$.
Since we have uniform convergence on any compact subset of
$\Omega \setminus \omega$, it is easy to show that $g$ also
is maximal plurisubharmonic on any such compactum, and therefore
on the whole of $\Omega \setminus \{0\}$. So $(dd^cg)^n = \mu \delta_0$,
and by Proposition \ref{massdecr}, 
$$
\mu = (dd^c g)^n (\Omega) \le \liminf_{\eps\to0} (dd^c G_{\mathcal
  I_\eps})^n (\Omega) = N= \ell (\mathcal I) \le e(\mathcal I)=(dd^c
G_{\mathcal I})^n (\Omega), 
$$
which implies the claimed inequality between the Monge-Amp\`ere measures.
\end{proof*}

\section{Convergence and Non-Convergence}
\label{secproofs}

\subsection{Proof of Theorem \ref{thmgci}}
\begin{proof*}

Since the family $(\|\Psi_\eps\|)_\eps$ is bounded on $\overline\Omega$,
there is a constant $C$ such that $\log \|\Psi_\eps\|-C\in PSH_-(\Omega)$, and it admits logarithmic singularities at the points $a_j^\eps$, thus
$\log \|\Psi_\eps\|-C\le G_\eps$. 

Since $\Psi_\eps$ has $n$ components, it is well-known that
$\log \|\Psi_\eps\|$ is a 
maximal plurisubharmonic function on $\Omega \setminus S_\eps$ (by composition
with the holomorphic map, it is enough to check it for $u(z):=\log\|z\|^2$,
an elementary computation).  If $a\in \Psi_0^{-1}\{0\}$, then for any small
enough $\delta>0$, $0< \min_{\partial B(a,\delta)} \| \Psi_0\|$, so if 
$\Psi_\eps^{-1}\{0\}\cap B(a,\delta)$ was empty, uniform convergence and 
maximality of $\log \|\Psi_\eps\|$ would imply that 
$$
\min_{\overline B(a,\delta)} \log \| \Psi_\eps\|> -1+\log (\min_{\partial B(a,\delta)} \| \Psi_0\|) 
>-\infty
$$
for $|\eps|$ small enough, which contradicts $\lim  \|\Psi_\eps(a)\|=0$. 

Since all the points of $ \Psi_\eps^{-1}\{0\}$ tend to $0$, we have 
$ \Psi_0^{-1}\{0\}=\{0\}$, and for any compact $K\subset \overline \Omega \setminus \{0\}$,
any $\eps$ close enough to $0$, $\min_K \|\Psi_\eps\|\ge c_K >0$, 
thus $G_\eps \le \log \| \Psi_\eps\| - \log c_K$. 
Since $G_S^\Omega (z) =-\infty$
iff $z \in S$, and $G_S(z) \le \log\|z-s\| +O(1)$ in a
neighborhood of each $s \in S$ \cite{De1}, \cite{Lel},
using \cite[Lemma 4.1]{Ra-Si}, we see that 
\begin{equation}
\label{geps}
G_\eps \le \log \|\Psi_\eps\|+O(1),
 \end{equation}
where the $O(1)$ is independent of $\eps$.

On the other hand, by the very definition of $\mathcal I_0$ and of
its Green function, $G_{\mathcal I_0}\ge \log \|\Psi_0\|+O(1)$. By
\cite[Theorem 2.5]{Ra-Si}, $G_{\mathcal I_0}\le \log \|\Psi_0\|+O(1)$
in a neighborhood of $0$, and again \cite[Lemma 4.1]{Ra-Si}
shows that 
\begin{equation}
\label{gzero}
G_{\mathcal I_0} \le \log \|\Psi_0\|+O(1) \mbox{ on } \Omega.
 \end{equation}
 
The hypothesis of uniform convergence of the $\Psi_\eps^j$ shows
that we can apply Lemma \ref{boundcv} with $G=\log \|\Psi_0\|$, therefore
$\lim_{\eps\to0}  G_\eps = g = \log \|\Psi_0\| +O(1)$,
with uniform convergence on compacta of $\overline \Omega \setminus \{0\}$. Furthermore, $g|_{\partial \Omega}=0$, and the uniform convergence
implies that $g$ is maximal plurisubharmonic on $\Omega \setminus \{0\}$,
so we can apply \cite[Lemma 4.1]{Ra-Si} in both directions to conclude
that $g= G_{\mathcal I_0}$, which proves the second statement in
Theorem \ref{thmgci}.

Now we prove the statement about ideals. For any $f \in \mathcal I_0$,
$$
f = \sum_{i=1}^n h_i \Psi_0^i = 
\lim_{\eps\to0} \sum_{i=1}^n h_i \Psi_\eps^i ,
$$
with uniform convergence on compacta, so that 
$\mathcal I_0 \subset \liminf_{\eps\to 0} \mathcal I_\eps$.

To prove the reverse inclusion, we need to use the characterization 
of an ideal $\mathcal I_\Psi :=\langle  \Psi^1, \dots,  \Psi^n \rangle$ by 
multidimensional residues. For simplicity, we assume that
$\Psi^{-1}(0)\cap \Omega = \{0\}$, so that a holomorphic function
belongs to $\mathcal I_\Psi$ if and only if its germ at $0$ is in the
ideal of germs with the same generators, $\mathcal I_{\Psi,0}$ (recall that
$\Omega$ is contractible and bounded). 
We take the next definition from \cite[\S 5.1, p. 14]{Ts}.

\begin{defn}
\label{residue}
Let $\Psi$ be a holomorphic mapping $\overline\omega \longrightarrow \C^n$
where $\omega$ is a bounded neighborhood of $0$ in $\C^n$, with 
$\Psi^{-1}(0)\cap  \overline\omega = \{0\}$. We choose a real
$n$-dimensional chain  
$$
\Gamma=\Gamma^\delta (\Psi):= 
\left\{ z \in \omega : |\Psi^j|=\delta_j, 1 \le j \le n \right\},
$$
where $\delta_j >0$ are small enough so that $\Gamma^\delta (\Psi)$
be relatively compact in $\omega$, with its orientation determined 
by the condition $d(\arg \Psi_1) \wedge \cdots \wedge d(\arg \Psi_n) \ge 0$
(by Sard's lemma, we can choose values of $\delta$ such that $\Gamma$
is smooth). 
Let $h$ be holomorphic
on $\omega$.
The \emph{local residue} of $h$ at 
the point $0$ is 
$$
\mbox{res}_{0,\Psi} (h) := \frac1{(2\pi i)^n} \int_\Gamma
\frac{h dz_1 \wedge \cdots \wedge dz_n}{\Psi^1 \cdots  \Psi^n}.
$$
\end{defn}

This residue is well-defined in the sense that it does not depend 
on the choice of a particular $\Gamma$ \cite[p. 15]{Ts}. In fact it
can be computed by integration on $\partial \omega$. We omit the
formula, but a consequence is the "continuity principle"
\cite[\S 5.4, Proposition, p. 20]{Ts}. We apply it to the 
situation of Theorem \ref{thmgci}.
By the hypothesis on $\Psi_\eps$, for 
$ |\eps| \le \eps_0$, $S_\eps=\Psi_\eps^{-1}(0)\subset\omega$
and is made up of isolated zeroes. 

\begin{prop}
\label{contpple}
Let $h$ be holomorphic in a neighborhood of $\overline \omega$, then
$$
\lim_{\eps\to0} \sum_{p \in S_\eps} \mbox{res}_{p, \Psi_\eps} h
= \mbox{res}_{0, \Psi_0} h.
$$
\end{prop}

The characterization of the ideal $\mathcal I_\Psi$ follows
immediately from the Local Duality Theorem \cite[\S 5.6, p. 23]{Ts}.

\begin{theorem}
\label{caracresid}
A germ $h$ of holomorphic function belongs to the ideal $\mathcal I_{\Psi,0}$ of germs of functions at $0$ generated by the components of $\Psi$ 
if and only if for any holomorphic germ $g$ at $0$, 
$\mbox{res}_{0,\Psi} (hg) =0$.
\end{theorem}

We return to the proof. Suppose that $f\in \mathcal O(\Omega)$
and there exists sequences $(f_j)\subset \mathcal O(\Omega)$
and $\eps_j\to 0$
such that $f_j\in \mathcal I_{\eps_j}$, $f_j \to f$
as $j \to\infty$, uniformly on compacta of $\Omega$. 
Take any holomorphic germ $g$ at $0$, and
 $\omega$  a neighborhood of $0$ small enough so that $g$ is 
 holomorphic in a neighborhood of $\overline \omega$ and we can apply
 Proposition \ref{contpple}. Then for any $p \in S_{\eps_j}$,
$\mbox{res}_{p, \Psi_{\eps_j}} f_j g =0$ by the "only if" part of
Theorem \ref{caracresid}, therefore 
$\mbox{res}_{p, \Psi_{0}} f g = \lim_{j\to\infty}0$ 
by Proposition \ref{contpple}, and the "if" part
of Theorem \ref{caracresid} shows that $f\in \mathcal I_0$. 
\end{proof*}

\subsection{Proof of Theorem \ref{thmiff}}
\begin{proof*}

In the ''only if" direction,
we prove a slightly stronger statement than in the Theorem: we only
assume that
$g=\lim_{\eps\to0} G_{\mathcal I_\eps}$ in $L^1_{loc}$, and
we don't
assume that $(\mathcal I_{\eps})_\eps$ converges. Write $\mathcal I=
\liminf_{\eps\to 0} \mathcal I_\eps$.  Then, by Lemma \ref{lengliminf}, 
$N\le \ell (\mathcal I)$. On the other hand,
by Proposition \ref{massdecr},
$
(dd^c g)^n (\Omega) \le  \liminf_{\eps\to0} (dd^c G_{\eps})^n(\Omega) =
N$. 

On the other hand, the Hilbert-Samuel multiplicity of an ideal,
which is invariant under integral closure
\cite[Theorem 10.4 (b)]{De2}, like its Green function, is
indeed related to it.

\begin{prop}
\label{massgreen}
If $V(\mathcal I) = \{a\}$, 
$(dd^c G_{\mathcal I})^n= e(\mathcal I) \delta_a$.
\end{prop}

A proof can be found for instance in \cite[Lemma 2.1, p. 4]{De3}.

Now suppose that 
$\ell (\mathcal I) < e(\mathcal I)$.
Then 
$$
(dd^c g)^n (\Omega) \le N\le \ell (\mathcal I)< e(\mathcal I)=(dd^c
G_{\mathcal I})^n(\Omega), 
$$
thus $g \neq G_{\mathcal I}$. 

Note that the same proof can be done when $\eps$ tends to $0$
along a subset. We omit the details. 

Conversely,
let $\psi_1,\ldots,\psi_n$ be generators of $\mathcal I$; we may
assume them to be defined on a neighborhood of
$\overline\Omega$. Since $\mathcal I_\eps\to\mathcal I$, for any
$\eps\in(0,\eps_0)$ there exist functions $\psi_{k,\eps}\in\mathcal
I_\eps$ converging to $\psi_k$ uniformly on compacts of $\Omega$,
$1\le k\le n$. Then by Rouch\'e's theorem (see e.g. \cite[\S 5.2,
Proposition 3, p. 16]{Ts}), each mapping
$\Psi_\eps=(\psi_{1,\eps},\ldots,\psi_{n,\eps})$ has isolated zeros,
say at points $a_j(\eps)$, $1\le j\le N(\eps)$, and their total
number, counted with the corresponding multiplicities $m_j(\eps)$,
equals the multiplicity of the mapping $\Psi=(\psi_1,\ldots,\psi_n)$
at $0$. The latter equals $N$ because having $n$ generators is
equivalent to the condition $e(\mathcal I)=\ell(\mathcal I)$, and
$\ell(\mathcal I)=\lim \ell(\mathcal I_\eps)=N$ by Lemmas
\ref{lenglimsup}, \ref{lengliminf}.  

By construction, $\mathcal I_{\Psi_\eps}\subseteq {\mathcal I}_\eps$, and  
$$
\ell({\mathcal I}_{\Psi_\eps})\le \sum_j \ell
\left( ({\mathcal I}_{\Psi_\eps})_{a_j(\eps)} \right) \le N = \ell(\mathcal I_\eps),
$$
so both ideals coincide. Therefore $G_\eps=G_{\mathcal I_\eps}= G_{\mathcal I_{\Psi_\eps}}$
 and the family $\mathcal I_{\Psi_\eps}$ satisfies the Uniform Complete Intersection condition, so an application of Theorem~\ref{thmgci} completes the proof.
\end{proof*}

\section{Example : two and three points systems}
\label{ex}

\subsection{The case of two points}

We begin by sketching what happens in the case where 
$S_\eps = \{p_1(\eps), p_2(\eps)\}$ (two distinct points), with 
$\lim_{\eps\to0} p_1(\eps) = \lim_{\eps\to0} p_2(\eps) =0\in \Omega$. 
By compactness, there is a subsequence $\eps_j$ such that 
$[p_1(\eps_j)- p_2(\eps_j)]\to \nu \in \mathbb P^{n-1} \C$,
where $[z]$ denotes the class of $z$ in $\mathbb P^{n-1} \C$, for $z\in \C^n\setminus \{0\}$.  From
now on we assume that $\eps \in \{\eps_j\}$ and drop the subscript.
By applying translations tending to $0$, we may assume 
that $p_1(\eps)=0$.

Pick orthonormal
bases $\mathcal B:= (v_1, \dots,v_n)$ such that $[v_1]=\nu$, and 
$\mathcal B^\eps:=(v_1^\eps, \dots,v_n^\eps)$ tending to $(v_1,
\dots,v_n)$ such that  
$[p_2(\eps)]= [v_1^\eps]$, so that $p_2(\eps)=\rho_\eps v_1^\eps$,
$\rho_\eps\to0$. 
Let $(\xi_1,\dots,\xi_n)$, resp.
$(\xi_1^\eps,\dots,\xi_n^\eps)$ stand for the coordinates in $\mathcal B$,
resp. $\mathcal B^\eps$. Finally, for $R>0$,
 $D_R := \{ z: \max|\xi_j| <R\}$, 
$D_R^\eps := \{ z: \max|\xi_j^\eps| <R\}$.

\begin{lemma}\cite{Ed}
\label{lemprod}
If $S_i \subset \Omega_i \subset \C^{n_i}$, $i= 1, 2$, then
$$
G_{S_1\times S_2}^{\Omega_1\times \Omega_2} (z_1,z_2) = 
\max \left( G_{S_1}^{\Omega_1} (z_1) , G_{S_2}^{\Omega_2} (z_2)\right) .
$$
\end{lemma}

Thus
$$
G_{S_\eps}^{D_R^\eps}(z) =
\max \left( \log\left|\frac{\xi_1^\eps(\rho_\eps-\xi_1^\eps)}{R^2-\xi_1^\eps \bar \rho_\eps}\right|,
\log\left|\frac{\xi_j^\eps}{R}\right|, 2\le j\le n
\right).
$$
There are $0<R_1<R_2$ such that for $|\eps|$ small enough, 
$D_{R_1}^\eps \subset \Omega \subset D_{R_2}^\eps$, so
$$
G_{S_\eps}^{D_{R_1}^\eps}(z) \ge 
G_{S_\eps}^{\Omega}(z)\ge G_{S_\eps}^{D_{R_2}^\eps}(z).
$$
The above inequality and an elementary computation show that, setting
$g(z):= \max (\log|\xi_1|^2, \log|\xi_j|,  2\le j\le n)$, there are constants $C_1,C_2\in \R$ such that if $z\in \Omega \setminus B(0,r_0)$
and $|\eps| < \eps_0(r_0)$, 
$$
g(z) +C_1 \le G_{S_\eps}^{\Omega}(z) \le g(z) +C_2.
$$
By Lemma \ref{boundcv}, $G_{S_\eps}^{\Omega}$ converges locally
uniformly on $\Omega \setminus \{0\}$ to $G=g+O(1)$.  This means in particular that $G(\zeta v_1)= 2\log|\zeta| + O(1)$, while for any
vector $v$ not colinear to $v_1$, $G(\zeta v)= \log|\zeta| + O(1)$. As a
consequence, if there are two subsequences $\{\eps_j\}$, $\{\tilde\eps_j\}$
such that $v_1 \neq \tilde v_1$ (with the obvious notation), then
$G_{S_\eps}^{\Omega}$ doesn't admit any limit, while if $\lim_{\eps\to0} [p_2(\eps)]$ exists (i.e. any convergent subsequence converges to the same
limit), then $\lim_{\eps\to0} G_{S_\eps}^{\Omega}=G=g+O(1)$, with local
uniform convergence. This necessary and sufficient condition for convergence
can be expressed by saying that $p_2^\eps$ converges in the blow-up at $0$
of $\C^n$.  It is the same as for convergence of two-point Lempert
functions \cite[Theorem 3.3]{Th-Tr2}. 
\vskip.5cm

\subsection{When all directions coincide}

\begin{proof*}{\it Proof of Theorem \ref{gen3pt}, part (2).}

Consider the following family of examples, where $\alpha \in \C$:
$$
S_\eps := \left\{ (0,0); (\eps, \alpha \eps^2) ;  (-\eps, \alpha \eps^2)\right\}.
$$
The ideal of functions vanishing on $S_\eps$ is clearly generated by 
$\{ z_1 (z_1^2-\eps^2) ,  z_2-\alpha z_1^2 \}$, so its limiting ideal 
is $< z_1^3,  z_2-\alpha z_1^2 >$ which depends on $\alpha$. 
This is a case where the Uniform Complete Intersection condition holds, and the 
limiting Green functions will be equivalent to $\max \left( 3 \log |z_1|, \log |z_2-\alpha z_1^2 |\right) $,
so they are distinct for distinct values of $\alpha$.
\end{proof*}

\subsection{Three points, two directions : limit ideal}

To prove the first part of Theorem \ref{gen3pt}, we first deal with the statement about convergence of ideals.  

Consider a function holomorphic in a neighborhood of the origin in $\C^n$ vanishing in two points $a, b$, close enough
to $0$.  Renormalizing if needed, we assume that $f$ is
holomorphic on the ball $B(a,1)$.  Let $d:=\|a-b\|$,
$v:=\frac1d (b-a)$, $f_{a,b}(\zeta):= f(a+\zeta v) \in \mathcal O(\overline \D)$. Then
$$
\int_{\partial \D} \frac{f_{a,b}(\zeta)}{\zeta(\zeta-d)}d\zeta=0,
$$
so if we have a sequence $f_\eps$ vanishing on $a_\eps,b_\eps\to 0$,
and converging uniformly on compacta to $f$, and if $v_\eps \to v$, then
$f(0,0)=0$ and
$$
\int_{\partial \D} \frac{f(\zeta v)}{\zeta^2}d\zeta=0,
$$
thus $\partial_v f(0,0)=0$. This is another special case of
the first statement of Theorem \ref{thmgci}.
In case (i) of Theorem \ref{gen3pt}, we may assume $i=3$, $j=2$
and 
perform this reasoning with $(a_\eps,b_\eps)= (a_1^\eps, a_2^\eps)$
and $(a_\eps,b_\eps)= (a_1^\eps, a_3^\eps)$;
the relevant vectors converge to 
two independent directions, $v_3$ and $v_2$ respectively,
we are in dimension  $2$, therefore all first
partial derivatives of $f$ must vanish at the origin.  This means that 
$\limsup_{\eps\to0} \mathcal I_\eps \subset \frak M_0^2$.  

To get the converse inclusion, it will be enough to approximate the generators of $\frak M_0^2$.  Since this ideal (as well as the 
notion of convergence) is invariant under invertible linear maps,
we may reduce ourselves to the case where $v_3=[1:0]$
and $v_2=[0:1]$.  Then $a_1^\eps=(x_1(\eps),x_2(\eps))$,
where $x_j(\eps)=o(1)$, $j=1,2$; and
$a_2^\eps-a_1^\eps= (\rho_2(\eps), \delta_2(\eps))$,
$a_3^\eps-a_1^\eps= (\delta_3(\eps), \rho_3(\eps))$, where
$\delta_j(\eps) =o(\rho_j(\eps))$, $j=2,3$. Then
$$
z_1^2 = \lim_{\eps\to0} (z_1-x_1(\eps)-\rho_2(\eps))
\left( z_1- x_1(\eps)- \frac{\delta_3(\eps)}{\rho_3(\eps)}(z_2- x_2(\eps))\right)
\in \liminf_{\eps\to0} \mathcal I_\eps,
$$
\begin{multline*}
z_1 z_2 = \lim_{\eps\to0} 
\left( z_1- x_1(\eps)- \frac{\delta_3(\eps)}{\rho_3(\eps)}(z_2- x_2(\eps))\right)
\left( z_2- x_2(\eps)- \frac{\delta_2(\eps)}{\rho_2(\eps)}(z_1- x_1(\eps))\right)\\
\in \liminf_{\eps\to0} \mathcal I_\eps,
\end{multline*}
and a similar computation takes care of $z_2^2$.

\subsection{First reduction : changing the axes}

The rest of this section is devoted to the more intricate proof
of the convergence of Green functions, which will be carried out by reduction to two model cases, given in the proposition below.
We use the notation from Definition \ref{defequiv}.

\subsubsection{Model Cases}
\begin{prop}
\label{model}
With the hypotheses and notations of Theorem \ref{gen3pt},
suppose furthermore that $\Om=\D^2$ and that there are functions
$\delta_i(\eps)$ with $\delta_i(\eps)/\eps \to 0$, $1\le i \le 4$
and $\omega(\eps) \to 0$ such that 
$$
a_1^\eps = (0,0), \quad a_2^\eps = (\eps, \delta_1(\eps)),
$$
and either
\begin{enumerate}
\item
$a_3^\eps = (\delta_2(\eps),\eps+\delta_3(\eps))$
\item
or $a_3^\eps = (\delta_4(\eps)\omega(\eps),\delta_4(\eps))$,
\end{enumerate}
then there exists plurisubharmonic functions $g_1$, $g_2$ such that 
in case (i), $i =1, 2$, 
$\lim_{\eps\to0} G_{\eps}=g_i$, uniformly on compacta of $\D^2  \setminus \{0\}$,
and $g_1 \sim_0 H$, $g_2\sim_0 F$, where $H$ and $F$ are described below. Consequently, $(dd^c g_i)^2 = 3 \delta_0$, $i=1,2$.
\end{prop}

We now  construct the auxiliary functions $H(z)$ and $F(z)$.
The following gives a partition of the extended complex plane : 
$$
\mathcal D_3 := \bar D(-1, \frac12), \quad \mathcal D_1 := \bar \D \setminus \mathcal D_3,
\quad \mathcal D_2 := \C \cup \{\infty\} \setminus (\mathcal D_1 \cup \mathcal D_3).
$$
We now define a partition of $\D^2  \setminus \{0\}$.
\begin{multline*}
\label{partition}
D_0 := \left\lbrace z \in \D^2  \setminus \{0\} : |z_2| \le |z_1|^2, \mbox{ or } |z_1| \le  |z_2|^2,
 \mbox{ or } |z_1+z_2| \le  |z_1|^2 
\right\rbrace  \\
D_j := \left\lbrace z \in \D^2  \setminus (D_0 \cup \{0\}) : \frac{z_2}{z_1} \in \mathcal D_j 
\right\rbrace, \quad j= 1, 2 , 3.
\end{multline*}
Here is the piecewise definition of our auxiliary function.
\begin{equation}
\label{Hformula}
H(z):= 
\begin{cases}
\max(2\log|z_1|, 2\log|z_2|) & \text{ if } z \in D_0,\\
\log|z_1| + \frac12 \log|z_2| & \text{ if } z \in D_1,\\
\frac12 \log|z_1| +  \log|z_2| & \text{ if } z \in D_2,\\
\log|z_1| + \frac12 \log|z_1+z_2| & \text{ if } z \in D_3.
\end{cases}
\end{equation}

The definition of the auxiliary function $F$
requires another partition of the bidisk.
\begin{equation*}
\label{partsing}
D'_0 := \left\lbrace z \in \D^2  \setminus \{0\} : |z_2| \le  |z_1|^{3/2} \mbox{ or } |z_1| \le |z_2|^2
\right\rbrace , \quad
D'_1 := \D^2  \setminus (D'_0 \cup \{0\}).
\end{equation*}
\begin{equation}
\label{Fformula}
F(z):= 
\begin{cases}
\max(2\log|z_1|, 2\log|z_2|) & \text{ if } z \in D'_0,\\
\frac12 \log|z_1| + \log|z_2| & \text{ if } z \in D'_1.
\end{cases}
\end{equation}

\subsubsection{Changing the coordinates}

To reduce ourselves to the cases occurring in Proposition \ref{model},
we will need to make linear changes of coordinates.

\begin{defn}
\label{defadm}
We say that a function $G : \Omega \to \R$ is \emph{admissible}
if and only if
\begin{equation}
\label{gooddecr}
\forall R>1, \exists C(R) \mbox{ such that }
\forall z \in \Omega \cap \frac1R \Omega, 
\left| G(z) - G(Rz) \right| \le C(R).
\end{equation}
\end{defn}

In particular, $G\sim_0 G'$ for any dilation $G'$ of $G$.  
When $G$ is unbounded, this forces it to have logarithmic growth.
We will
see that this onerous looking hypothesis is actually satisfied in
the examples we are interested in.

\begin{lemma}
\label{lincv}
Let $\Omega$ be a bounded hyperconvex domain, $0\in \Om$.
Suppose that $(S_\eps)\subset \Omega$ is  such that 
for any $R>0$, $\lim_{\eps\to0} G_{RS_{\eps}} (z)  = g_R(z)$, uniformly on compacta of
$\Omega \setminus \{0\}$, and that there exists a function $G$ on $\Omega$
such that for any $R>0$, $g_R\sim_0 G$, and $G$ is admissible. 
Let $\Phi$ be an invertible linear map of $\C^n$, and 
$\Om_1$ a domain satisfying the same hypotheses as $\Om$.

 Then $\lim_{\eps\to0} G^{\Om_1}_{\Phi^{-1}(S_{\eps})} (z)  = \tilde g(z)$, uniformly on compacta of
$\Omega_1 \setminus \{0\}$, with $\tilde g(z) \sim_0 g \circ \Phi$. 
\end{lemma}

\begin{proof}
Let $0<R_1<R_2$ be such that 
$R_1 \Omega \subset \Phi(\Omega_1) \subset R_2 \Omega$.
Since $G^{\Omega_1}_{\Phi^{-1}(S_{\eps})} (z)=G^{\Phi(\Omega_1)}_{S_{\eps}} (\Phi(z))$, we have
$$
G^{\Omega}_{R_2^{-1}S_{\eps}} (\frac1{R_2}\Phi(z))
=
G^{R_2\Omega}_{S_{\eps}} (\Phi(z))
\le
G^{\Omega_1}_{\Phi^{-1}(S_{\eps})} (z) 
\le
G^{R_1\Omega}_{S_{\eps}} (\Phi(z))
=
G^{\Omega}_{R_1^{-1}S_{\eps}} (\frac1{R_1}\Phi(z)).
$$
Let $\delta>0$. On $\Omega \setminus B(0,\delta)$, for $\eps$
small enough,
$$
G^{\Omega_1}_{\Phi^{-1}(S_{\eps})} (z)  \ge g_{R_2^{-1}} (\frac1{R_2}\Phi(z)) -1
\ge G(\frac1{R_2}\Phi(z)) -C_1 \ge G(\Phi(z)) -C_2,
$$
and
$$
G^{\Omega_1}_{\Phi^{-1}(S_{\eps})} (z)  \le g_{R_1^{-1}} (\frac1{R_1}\Phi(z)) +1
\le G(\frac1{R_1}\Phi(z)) +C_3 \le G(\Phi(z)) +C_4.
$$
So the hypotheses of Lemma \ref{boundcv} are satisfied, with the function
$G\circ \Phi$ instead of $G$.
\end{proof}

\subsubsection{Perturbations}

The following lemma and its corollary will be useful 
to get rid of higher order
perturbations of the family $(S_\eps)$.

\begin{lemma}
\label{perturb}
Let $\Omega$ be a bounded hyperconvex domain such that the 
function $G^\Omega_{0}$ is Lipschitz in a neighborhood of $\partial \Omega$.
There exists $\delta_0>0$ such that the following holds for any $\delta \in (0,\delta_0]$.

Let $S, S' \subset B(0,\delta) \subset \subset \Omega$,
$\#S=\#S'=N$. Suppose there exists a biholomorphic map $\Phi $ from
$\Omega$ to $\Phi (\Omega)$ such that $\Phi(S)=S'$, and 
$\|\Phi(z)-z\|\le \delta$ for
any $z\in \Omega$ such that $\mbox{dist}(z,\partial \Omega) \ge \delta$.

Then there are constants $C_1, C_2 >0$ depending only on $N$ and $\Omega$
such that for any $z$ verifying $G^\Omega_{\{0\}}(z) \le -C_1 \delta$,
$$
G^\Omega_{S'} (\Phi(z)) \le G^\Omega_{S} (z) + C_2 \delta.
$$
\end{lemma}

\begin{proof}
Let
$$
\Omega_\delta := \left\lbrace z \in \Omega : 
G^\Omega_{0} (z) < \inf_{\mbox{dist}(z,\partial \Omega) \le \delta} G^\Omega_{0} \right\rbrace .
$$
Since $\Omega_\delta \subset \{  z \in \Omega :  \mbox{dist}(z,\partial \Omega) \ge \delta \}$,
$\Phi (\Omega_\delta) \subset \Omega$.

Since $G^\Omega_{0}$ is Lipschitz in a neighborhood of $\partial \Omega$, 
$\Omega_\delta \supset \{ z \in \Omega : 
G^\Omega_{0} (z) \le -C_1 \delta \}$.

By Lemma \ref{roughest}, for $\delta$ small enough,
on $\partial \Omega_\delta$
$$
2N C_1 \delta + G^\Omega_{S} (z) \ge 
-2N G^\Omega_{0} (z)  + G^\Omega_{S} (z) 
\ge 0  .
$$
So the plurisubharmonic functions $G^\Omega_{S'} \circ \Phi $ and $G^\Omega_{S}  + 2N C_1 \delta$
have the same singularities on the set $S$, the first one is negative and the second one 
positive on $\partial \Omega_\delta$, and the second one is maximal psh; by \cite[Lemma 4.1]{Ra-Si}, we have 
$$
G^\Omega_{S'} \circ \Phi (z) \le G^\Omega_{S} (z)  + 2N C_1 \delta,
$$
for all $z\in \Omega_\delta$, q.e.d.
\end{proof}

\begin{cor}
\label{corpert}
Under the hypotheses of Lemma \ref{perturb}, suppose that we have 
a family $(S_\eps)$ such that $\lim_{\eps\to0} G_{S_{\eps}} (z)  = g(z)$, uniformly on compacta of
$\Omega \setminus \{0\}$, and a family of biholomorphic mappings $\Phi_\eps(z)= z + \Theta_\eps(z)$,
with $\lim_{\eps\to0} \|\Theta_\eps\|_L =0$, where $\|\cdot\|_L$ denote the Lipschitz norm
($\|F\|_L:= \sup_{z\in\Omega }\|F(z)\| + \sup_{z,w\in\Omega } \frac{\|F(z)-F(w)\| }{\|z-w\| }$).

Then $\lim_{\eps\to0} G_{\Phi_{\eps}(S_{\eps})} (z)  = g(z)$, uniformly on compacta of
$\Omega \setminus \{0\}$.
\end{cor}
\begin{proof}
Given $K\subset \subset \Omega \setminus \{0\}$ and $\eta >0$, choose $K_1$
such that $K\subset \subset K_1\subset \subset \Omega \setminus \{0\}$. 
For $\eps$ small enough, $\Phi_\eps (K) \subset K_1$ and $\Phi_\eps (K_1) \supset K$.

Choose $\eps$ small enough so that for all $z\in K_1$, $|G_{S_{\eps}} (z) - g(z)| <\eta/4$. 
By Lemma \ref{perturb} applied to $\Phi_\eps$ and to $\Phi_\eps^{-1}$, 
for $\eps$ small enough,
$$
G_{S_{\eps}} (\Phi_\eps^{-1}(z)) - \eta/4 \le G_{\Phi_{\eps}(S_{\eps})} (z) \le 
G_{S_{\eps}} (\Phi_\eps^{-1}(z)) + \eta/4.
$$
Because of the continuity of each $G_\eps$ \cite{Lel} and of the uniform
convergence, $g$ is (uniformly) continuous on $K_1$, so that for $\eps$
small enough, $z\in K_1$, $|g(z)-g(\Phi_\eps^{-1}(z)) |< \eta/4$. Putting all the inequalities
together we get $| G_{\Phi_{\eps}(S_{\eps})} (z) -g(z)|<\eta$ for $z\in K$.
\end{proof}

\subsubsection{Reduction to the Model Cases}

\begin{proof*}{\it Proof of Theorem \ref{gen3pt}, part (1).}

By applying Corollary \ref{corpert} to the translations 
$\Phi_\eps(z)=z-a_1^\eps$, we see that we may assume $a_1^\eps=0$
for all $\eps$. 

We need to check that we may apply Lemma \ref{lincv}.

\begin{lemma}
\label{rdiff}
The functions $H$  and $F$ defined after
Proposition \ref{model} are admissible.
\end{lemma}

\begin{proof}
When there is an $i \in \{0,1,2,3\}$ such that both $z$ and $Rz \in D_i$, 
the property follows immediately from the formula for $H$. Otherwise, in
cases where (for instance) $z \in D_1$ and $Rz \in D_0$, it is a straighforward
computation. The constants $C(R)$ are fixed multiples of $\log R$.  The computation is similar (if anything, easier) for $F$.
\end{proof}

In the generic case where
 all the $v_i$ in the statement of Theorem \ref{gen3pt} are distinct, we choose an invertible linear map $\Phi$
such that 
$[\Phi (\tilde v_3)]=[1:0]$, $[\Phi (\tilde v_2)]=[0:1]$, $[\Phi (\tilde v_1)]=[1:-1]$ 
(this follows from an elementary computation with matrices, or from the well-known fact that M\"obius maps on the Riemann sphere are 
transitive on triples of points). By Lemma \ref{lincv}, we can 
reduce ourselves to studying $\Phi(S_\eps)$, which means that we take
\begin{equation}
\label{sqpts}
a_1^\eps = (0,0), \quad
a_2^\eps = (\rho_2(\eps),\eta_2(\eps)), \quad
a_3^\eps = (\eta_3(\eps),\rho_3(\eps)),
\end{equation}
with all coordinates tending to $0$ and, given the new
values of $v_3$ and $v_2$,
$\lim_{\eps\to0} \eta_j(\eps)/\rho_j(\eps)=0$, $j=2,3$. 

If we set 
$\gamma(\eps):= \frac{\eta_3(\eps)-\rho_2(\eps)}{\rho_3(\eps)-\eta_2(\eps)}$, then $v_1^\eps= [\gamma(\eps):1]$, so $\lim_{\eps\to0} \gamma(\eps)=-1$.
Hence 
$$
\lim_{\eps\to0} \frac{\rho_2(\eps)}{\rho_3(\eps)}
= \lim_{\eps\to0} \frac{\frac{\eta_3(\eps)}{\rho_3(\eps)}-\gamma(\eps)}{1- \gamma(\eps) \frac{\eta_2(\eps)}{\rho_2(\eps)}} = 1.
$$
Now denote by $\eps$ what was denoted $\rho_2(\eps)$ (if the application
was not one-to-one, there is some ambiguity in our new notations, but all possibles choices will give subfamilies converging to the same limit). Then $\eta_2(\eps)$ becomes $\delta_1(\eps)$, 
$\eta_3(\eps)$ becomes $\delta_2(\eps)$, 
and $\rho_3(\eps)$
becomes $\eps (1+o(1))= \eps + \delta_3(\eps)$. We are reduced to case (1) of Proposition \ref{model}.

On the other hand, in the degenerate case where
 $v_1=v_3\neq v_2$, we can still take $\Phi$
such that $[\Phi (\tilde v_3)]=[1:0]=[\Phi (\tilde v_1)]$, $[\Phi (\tilde v_2)]=[0:1]$.
Then \eqref{sqpts} and the 
limits following it still hold.  But, setting $\gamma(\eps)$ as before,
we have $\lim_{\eps\to0} \gamma(\eps)=\infty$, so
$$
\lim_{\eps\to0} \frac{\rho_2(\eps)}{\rho_3(\eps)}
= \lim_{\eps\to0} \frac{\gamma(\eps)^{-1}\frac{\eta_3(\eps)}{\rho_3(\eps)}-1}{\gamma(\eps)^{-1}-  \frac{\eta_2(\eps)}{\rho_2(\eps)}} = \infty.
$$
Again, denote by $\eps$ what was denoted $\rho_2(\eps)$, then 
$\rho_3(\eps)=o(\rho_2(\eps))$ becomes  $\delta_4(\eps)$
and $\eta_3(\eps)=o(\rho_3(\eps))$ becomes  $\delta_4(\eps)\omega(\eps)$; we are reduced to case (2) of Proposition \ref{model}.

The properties about the decrease of $g$ in Theorem \ref{gen3pt} 
are invariant under linear changes of coordinates, and are easily verified on the functions $F$ and $H$.
\end{proof*}

\subsection{Simplified model cases}

\begin{proof*}{\it Proof of Proposition \ref{model}.}

We apply Corollary \ref{corpert} with the map 
$$
\Phi_\eps(z)= (z_1,z_2) + \left( \frac{\delta_2(\eps)}\eps z_2, \frac{\delta_1(\eps)}\eps z_1 + \frac{\delta_3(\eps)}\eps z_2 \right),
$$
which transforms the system $ \left\{ (0,0), (\eps,0), (0,\eps)\right\}$
into the one given in case (1) of  Proposition \ref{model}. 
The proof of that case reduces to:

\begin{lemma}[Generic case]
\label{3sym}
Suppose that $S_\eps = \left\{ (0,0), (\eps,0), (0,\eps)\right\}$. Then
$\lim_{\eps\to0} G_{\eps}=g_H$, uniformly on compacta of $\D^2  \setminus \{0\}$,
and $g_H \sim_0 H$. 
\end{lemma}

For case (2), we use the map 
$$
\Phi_\eps(z)= (z_1,z_2) + \left( \omega(\eps) z_2, \frac{\delta_1(\eps)}\eps z_1  \right),
$$
which maps the system $S_\eps = \left\{ (0,0), (\eps,0), (0,\delta_4(\eps))\right\}$ to the one given in case (2).  We exchange the axes,
setting $\check{F}(z_1,z_2)= F(z_2,z_1)$. For simplicity of notation,
we also write
$\rho:=\delta_4$. The proof reduces to:

\begin{lemma}[Degenerate case]
\label{3sing} 
Let $S_\eps = \left\{ (0,0), (\rho(\eps), 0), (0, \eps)\right\}$,
where  $\lim_{\eps\to0} \rho(\eps)/\eps=0$. Then 
$\lim_{\eps\to0} G_{\eps}=g_F$, uniformly on compacta of $\D^2  \setminus \{0\}$,
and $g_F\sim_0 \check{F}$. 
\end{lemma}
\end{proof*}

\subsection{Lower estimate}

Now we start the proof of Lemmas \ref{3sym} and \ref{3sing}. 

Let $\psi_\eps (z) := z_1 + \frac{\rho(\eps)}{\eps} z_2 - \rho(\eps)$ (where 
$\rho(\eps)=\eps $ or $\rho(\eps)=o(\eps) $).  Let
$$
L_\eps (z) = \max \left(  \frac12 \log \left| z_1 z_2 \psi_\eps (z) \right| , 
\log  \left| z_1 \frac{\rho(\eps) -z_1}{1 - \overline{\rho(\eps)} z_1} \right|,
\log  \left| z_2 \frac{\eps -z_2}{1 - \overline{\eps} z_2} \right| 
\right) - \frac12 \log 3 .
$$
\begin{lemma}
\label{lower}
Suppose that $|\rho(\eps)| \le |\eps|$ for any $\eps$, then
$$
G_\eps (z) \ge L_\eps (z), \mbox{ for any } z \in \D^2.
$$
Furthermore, $\lim_{\eps\to 0} L_\eps$ exists uniformly on any
compact subset of $\D^2 \setminus \{0\}$, and
\begin{enumerate}
\item
if $\rho(\eps) = \eps$, then $\lim_{\eps\to 0} L_\eps \sim_0 H$;
\item
if $\rho(\eps) = o(\eps)$, then $\lim_{\eps\to 0} L_\eps \sim_0 \check{F}$.
\end{enumerate}
\end{lemma}

\begin{proof}
Let $S'_\eps := S_\eps \cup \{ (\rho(\eps), \eps)\}$. This is a product set
and by Lemma \ref{lemprod}, we get 
$$
G_\eps (z) \ge G_{S'_\eps} (z) = 
\max \left(  
\log  \left| z_1 \frac{\rho(\eps) -z_1}{1 - \overline{\rho(\eps)} z_1} \right|,
\log  \left| z_2 \frac{\eps -z_2}{1 - \overline{\eps} z_2} \right| 
\right).
$$
At each of the three points of $S_\eps$, exactly two of the three holomorphic
functions $z_1$, $ z_2$ and $ \psi_\eps (z)$ vanish, and $|z_1 z_2 \psi_\eps (z)|\le 3$
for all $z \in \D^2$, so 
$$
G_\eps (z) \ge 
\frac12 \log \left| z_1 z_2 \psi_\eps (z) \right| - \frac12 \log 3,
$$
because the right hand side is a negative plurisubharmonic function with
the correct singularities. 

It is easy to see that 
$$
\lim_{\eps\to 0} L_\eps (z) = \max \left(  \frac12 \log \left| z_1 z_2 (z_1+z_2) \right| , 
2 \log  | z_1|, 2 \log  | z_2|
\right) - \frac12 \log 3
$$ 
if $\rho(\eps) = \eps$, and
$$
\lim_{\eps\to 0} L_\eps (z) = \max \left(  \frac12 \log | z_1^2 z_2 | , 
2 \log  | z_1|, 2 \log  | z_2|
\right) - \frac12 \log 3 
$$
if $\rho(\eps) = o(\eps)$, with uniform convergence on any
compact subset of $\D^2 \setminus \{0\}$. 

We give only the estimate from below of  $\lim_{\eps\to 0} L_\eps $ 
in the inequalities implicit in (1) and (2) (those are the only ones needed to
prove Lemmas \ref{3sym} and \ref{3sing}). 

First we deal with statement (1) in the Lemma.

For $z \in D_0$, the inequality is immediate. For $z \in D_1$, 
$\left| \frac{z_2}{z_1} +1\right| \ge \frac12 $, so $\left| z_1 z_2 (z_1 + z_2) \right| \ge | z_1 z_2| \frac12  | z_1| $,
and we get 
$$ 
\lim_{\eps\to 0} L_\eps (z) \ge \frac12 \log \left| z_1 z_2 (z_1+z_2) \right| - \frac12 \log 3 \ge H(z) - \frac12 \log 6. 
$$

By considering the image of the disk $\mathcal D_3$ under the inversion $\zeta \mapsto 1/\zeta$,
we see that $\left| \frac{z_2}{z_1} +1\right| \ge \frac12 $ implies $\left| \frac{z_1}{z_2} +1\right| \ge \frac13 $,
so for $z \in D_2$, $ \frac12 \log \left| z_1 z_2 (z_1+z_2) \right| \ge \frac12 \log | z_1 z_2^2| - \frac12\log 3= H(z) - \frac12 \log 3$.

Finally, for $z \in D_3$, 
$\left| \frac{z_2}{z_1} +1\right| \le \frac12 $, so $ |z_2| \ge \frac12  |z_1|$ and 
$ \frac12 \log \left| z_1 z_2 (z_1+z_2) \right| \ge \frac12 \log | z_1^2 (z_1+z_2)| - \frac12\log 2= H(z) - \frac12 \log 2$.

In case (2), the computations are even easier. 
\end{proof}

\subsection{Upper estimates for the generic case}

We will now estimate $G_\eps$ from above by constructing certain analytic discs.  This should be compared
with \cite{Th}.

In what follows, $S_\eps$ and $G_{\eps}$ are as in Lemma \ref{3sym}.

\subsubsection{Near the coordinate axes}

\begin{lemma}
\label{aboveD0}
Let $\delta >0$. For any $\eta>0$
there exists $m=m(\delta, \eta)>0$ such that for any $\eps \in D(0,m)$,
for any $z_1,z_2 \in D_0 \setminus D(0,\delta)^2$,
$$
G_{\eps} (z_1,z_2) \le 2 \log \left( \max (|z_1|, |z_2|) \right) + \eta.
$$
\end{lemma}

\begin{proof}
Suppose for instance that $|z_2| \le |z_1|^2$. Then $\max (|z_1|, |z_2|)=|z_1|$
and $\delta \le |z_1|$.
Let
$$
\tilde \varphi_\eps (\zeta) := \left( \zeta , \frac{\zeta(\zeta-\eps)}{z_1(z_1-\eps)} z_2 \right) .
$$
Since 
$$
\left|  \frac{z_2}{z_1(z_1-\eps)}  \right| = \left|  \frac{z_2}{z_1^2}  \frac1{1-\frac{\eps}{z_1}}\right|
\le \frac1{\left|1-\frac{|\eps|}{\delta}\right|},
$$
we can choose $0<\gamma = O(|\eps|)$ such that 
$\varphi (\zeta ) := \tilde \varphi \left( \frac{\zeta}{1+\gamma} \right)$
defines a map from $\D$ to $\D^2$.

Now $G_{\eps} \circ \varphi \in SH_-(\D)$, and it has logarithmic singularities 
at $0$ and $(1+\gamma)\eps$, so (using the explicit formula for the Green function
in the unit disk)
$$
G_{\eps} (z_1,z_2) = G_{\eps} \circ \varphi ((1+\gamma)z_1) 
\le \log \left| (1+\gamma)z_1 \frac{(1+\gamma)(z_1-\eps)}{1- |1+\gamma|^2 z_1 \bar \eps} \right|,
$$
which yields the required inequality for $|\eps| <m$. 

If $|z_1| \le |z_2|^2$, we just exchange the roles of the coordinates. If 
$|z_1+z_2| \le |z_1|^2$, we perform an analogous computation with 
$$
\tilde \varphi_\eps (\zeta) := \left( \zeta , (\zeta-\eps)(\alpha \zeta -1) \right) ,
\quad
\alpha := \frac{z_1+z_2 - \eps}{z_1(z_1-\eps)} .
$$
\end{proof}

\subsubsection{Analytic discs for the most common case}

Away from the exceptional region $D_0$, the construction of the analytic disks is more delicate.

\begin{lemma}
\label{above3sym}
Let $\delta >0$.
Let $(z_1,z_2) \in \D^2 \setminus (D_0 \cup D(0,\delta)^2)$, $\left| \frac{z_1}{z_2} +1\right| \ge \frac12$. 

For any $\eta>0$ there exists $m>0$ such that for any $\alpha \in D(0,m)$,
there exists $\varphi \in \mathcal O(\D,\D^2)$ and
$\zeta_1, \zeta_2 \in D(0,\eta)$, $\zeta_3 \in \D $ such that 
$\left| |\zeta_3|^2 - \max (|z_1|,|z_2|) \right| \le \eta$ and
$$
\varphi(0) = (0,0), 
\varphi (\zeta_1) = (\alpha^2, 0),
\varphi (\zeta_2) = (0,\alpha^2),
\varphi (\zeta_3) = (z_1,z_2).
$$
\end{lemma}

\begin{proof}
Case 1. $|z_2|^2 \le |z_1|  \le |z_2| $.

Let $\xi_1, \xi_2 \in D(0,\frac12) \subset \C$. 
We define a mapping from $\D$ to $\C^2$ by
\begin{equation}
\label{approxdisk}
\tilde \varphi_\xi (\zeta) := 
\left(  \frac{z_1}{z_2} (1+\xi_1) \zeta (\zeta - \tilde \zeta_2), (1+\xi_2) \zeta (\zeta - \tilde \zeta_1) \right) , 
\end{equation}
with $ \tilde \zeta_1,  \tilde \zeta_2$ to be defined below. 
We note that $\tilde \varphi_\xi (0) = (0,0)$, 
and
\begin{eqnarray}
\label{values}
\tilde \varphi_\xi (\tilde \zeta_1) &= & \left(  \frac{z_1}{z_2} (1+\xi_1) \tilde \zeta_1 (\tilde \zeta_1 - \tilde \zeta_2),0
\right) , \\
\label{values'}
\tilde \varphi_\xi (\tilde \zeta_2) &= & \left(  0,  (1+\xi_2) \tilde \zeta_2 (\tilde \zeta_2 - \tilde \zeta_1)
\right).
\end{eqnarray}

We will need the auxiliary quantity
$$
\mu := \alpha \left( \frac1{1+\xi_2} + \frac1{1+\xi_1} \frac{z_2}{z_1} \right)^{1/2},
$$
where we choose one particular square root of $1+ \frac{z_2}{z_1}$, and define the 
above square root to be continuous in a neighborhood of $\xi_1=0, \xi_2=0$ (we 
may need to reduce the size of the disk where $\xi_1, \xi_2$ can be chosen).

Note that $\left| \frac{z_1}{z_2} +1\right| \ge \frac12$ implies
$\left| \frac{z_2}{z_1} +1\right| \ge \frac13$, and that 
$ \left| \frac{z_2}{z_1} \right| = \left| \frac{z_2^2}{z_1} \frac1{z_2}\right| \le \frac1\delta $,
so that for $\xi_1, \xi_2$ small enough,
$$
\frac14 |\alpha| \le |\mu| \le \frac2{\sqrt \delta} |\alpha| .
$$

We now set 
$$
\tilde \zeta_1 := \frac{\alpha^2}{(1+\xi_1)\mu} \frac{z_2}{z_1} , \quad
\tilde \zeta_2 :=\tilde \zeta_1 - \mu = -  \frac{\alpha^2}{(1+\xi_2)\mu}.
$$

Substituting this into \eqref{values},  \eqref{values'}, we get
\begin{eqnarray}
\tilde \varphi_\xi (\tilde \zeta_1) &= & \left(  \frac{z_1}{z_2} (1+\xi_1) \tilde \zeta_1 \mu ,0\right) 
= (\alpha^2, 0), \\
\tilde \varphi_\xi (\tilde \zeta_2) &= & \left(  0,  - (1+\xi_2) \tilde \zeta_2 \mu \right)
=  (0,\alpha^2).
\end{eqnarray}

We have an analytic disk passing through all three poles, now we need to have it go through 
the point $z$. Let $\tilde \zeta_3 := z_2^{1/2}$ (it doesn't matter which square root we choose).
Then
\begin{multline*}
\tilde \varphi_\xi (\tilde \zeta_3) - (z_1,z_2) 
= \left( z_1 \xi_1 + \frac{(1+\xi_1)\alpha^2}{(1+\xi_2)\mu} \frac{z_1}{z_2^{1/2}} ,
z_2 \xi_2 - \frac{(1+\xi_2)\alpha^2}{(1+\xi_1)\mu} \frac{z_2^{3/2}}{z_1} \right) \\
=: \left( z_1 \left( \xi_1- \Phi_1(\xi,\alpha) \right) ,  z_2 \left( \xi_2- \Phi_2(\xi,\alpha) \right) \right) .
\end{multline*}
Since $\alpha^2/\mu=O(|\alpha|)$ and $\left| \frac{z_1}{z_2^{1/2}} \right|
\le 1/\eta$, $\left| \frac{z_2^{3/2}}{z_1} \right|
\le 1/\eta^2$,
for $|\alpha|$ small enough, the map $\xi \mapsto \Phi (\xi,\alpha) $ is contractive,
so there exists a unique $\xi^0$ such that $\xi^0= \Phi (\xi^0,\alpha)$, and
$\|\xi^0\| \le C \delta^{-1} |\alpha|$. Let $\tilde \varphi (\zeta):= \tilde \varphi_{\xi^0} ( \zeta)$. 

The map  $\tilde \varphi$ hits the correct points, but in general $ \tilde \varphi (\D)\not\subset \D^2$.
However, the estimates on $\mu$ and $\xi^0$ imply that  
$$
\tilde \varphi (\zeta) = \left( \frac{z_1}{z_2}\zeta^2, \zeta^2\right) + O(\delta^{-1} |\alpha|),
$$
so using the Schwarz Lemma there exists $0<\gamma = O(\delta^{-1} |\alpha|)$ such that 
letting
$$
\varphi (\zeta ) := \tilde \varphi \left( \frac{\zeta}{1+\gamma} \right) , \quad
\zeta_j := (1+\gamma) \tilde \zeta_j, 1\le j \le 3,
$$
we have a map and points in the disk satisfying the conclusions of the Lemma for 
$|\alpha|$ small enough. 

Case 2. $|z_1|^2 \le |z_2|  \le |z_1| $.

The computations are almost exactly the same once we exchange $z_1$ and $z_2$.
\end{proof}

\subsubsection{Upper estimate from the analytic disks}

It is easy to deduce from the previous Lemma a rough global estimate on the function
$G_{\eps}$.

\begin{lemma}
\label{roughestabove}
Then there exists $C>0$ such that for any
$\delta >0$, there is an $m=m(\delta)>0$ such that for any $\eps \in D(0,m)$, 
 for any $(z_1,z_2) \in \D^2 \setminus (D_0 \cup D(0,\delta)^2)$,
$$
G_{\eps} (z_1,z_2) \le \frac32 \log \left( \max (|z_1|, |z_2|) \right) + C .
$$
\end{lemma}

Notice that for $(z_1,z_2) \in D_0 \setminus D(0,\delta)^2$, Lemma \ref{aboveD0}
gives an even better estimate, so that the result actually holds on the whole of
$\D^2 \setminus  D(0,\delta)^2$.

\begin{proof}
If $\left| \frac{z_1}{z_2} +1\right| \ge \frac12$, then by choosing an analytic
disk $\varphi$ as in Lemma \ref{above3sym} with $\alpha = \eps^{1/2}$, 
we have $G_{\eps} \circ \varphi \in SH_- (\D)$, with three logarithmic poles,
and a reasoning similar to the proof of Lemma \ref{aboveD0} shows that
$G_{\eps} (z) \le \frac32 \log \left( \max (|z_1|, |z_2|) \right) + 1$
when  $\eta $ is small enough, and therefore when $|\eps|$ is small enough.

To deal with the remaining case, consider the affine biholomorphism
$$
L (z_1,z_2) := (z_1, \eps -z_1-z_2).
$$
It sends $S_{3,\eps} $ to itself (changing the roles of the points) and 
the set $\left\lbrace \left| \frac{z_1}{z_2} +1\right| < \frac12 \right\rbrace $
to $\left\lbrace \left| \frac{w_1}{w_2-\eps} +1\right| > 2 \right\rbrace $.  Also,
$L^{-1}(\D^2) \subset 3 \D^2$ for $|\eps|\le 1$. So
$$
G_{\eps}^{\D^2} (\frac{z}3) = G_{3\eps}^{3\D^2} (z) \le  G_{3\eps}^{L^{-1}(\D^2)} (z) 
= G_{3\eps}^{\D^2} (L(z) ),
$$
therefore for $|\eps|$ small enough with respect to $\delta$,
we have $\left| \frac{L_1(3z)}{L_2(3z)}+1\right| > \frac12$ and
\begin{multline*}
G_{\eps}^{\D^2} (z) \le G_{3\eps}^{\D^2} (L(3z) ) 
\\
\le 
\frac32 \log \left( \max (|L_1(3z)|, |L_2(3z)|) \right) + 1
\le  \frac32 \log \left( \max (|z_1|, |z_2|) \right) + C.
\end{multline*}
. 
\end{proof}

\subsubsection{End of proof of Lemma \ref{3sym}}
\begin{proof*}

Any compact subset of $ \D^2\setminus \{0\}$ is contained in
$ \D^2\setminus D(0,\delta)^2$ for some $\delta >0$.

Case 1. $|z_1|^2 \le |z_2|  \le  |z_1| $.

Consider the holomorphic map from $\D$ to $\D^2$ given by $\varphi (\zeta) = (\zeta, \frac{z_2}{z_1} \zeta)$,
and $u:=G_{\eps}\circ \varphi \in SH_-(\D)$. By Lemma \ref{roughestabove}, for any
$\delta_1>0$ and $\zeta$ such that $|\zeta|=\delta_1$, $|\eps|<m(\delta_1)$, we have $u(\zeta) \le \frac32 \log \delta_1 +C$.
On the other hand, when $|z_2|  <  |z_1|$, if $|\zeta|= \left| \frac{z_2}{z_1} \right|$, then $\varphi (\zeta) \in
\{ w : |w_2|\le |w_1|^2\}$ and by Lemma \ref{aboveD0}, for any $\eta_1>0$ we can choose $|\eps|$ small
enough so that 
$$
u(\zeta) \le 2 \log \left| \frac{z_2}{z_1} \right| + \eta_1.
$$
Applying the three-circle theorem, we get, for any $\zeta \in D(0, \left| \frac{z_2}{z_1} \right|)\setminus D(0, \delta_1 )$,
\begin{equation}
\label{threecircineq}
u(\zeta) \le 
\frac{\log \left| \frac{z_2}{z_1} \right| - \log|\zeta|}{\log \left| \frac{z_2}{z_1} \right|- \log \delta_1}  
\left( \frac32 \log \delta_1 +C \right)  
+ 
\frac{\log|\zeta|- \log \delta_1}{\log \left| \frac{z_2}{z_1} \right|- \log \delta_1}  
\left( 2 \log \left| \frac{z_2}{z_1} \right| + \eta_1 \right) .
\end{equation}
This inequality still holds in the case $|z_2|  = |z_1| $, simply using the fact that $u\le 0$.
We may apply it for $\zeta = z_1$. Let $\eta_2>0$. Then 
$$
\frac{\log \left| \frac{z_2}{z_1} \right| - \log|z_1|}{\log \left| \frac{z_2}{z_1} \right|- \log \delta_1}  
\left(  \frac32 \log \delta_1 +C \right)  
=\left( \log |z_2| - 2\log|z_1| \right) 
\frac{-\frac32 + \frac{C}{|\log \delta_1|}}{1-\frac{|\log \left| \frac{z_2}{z_1} \right||}{|\log \delta_1|} },
$$
and for $\delta_1$ small enough (depending on $\eta_2$ and $\delta$),
$$
\frac{-\frac32 + \frac{C}{|\log \delta_1|}}{1-\frac{|\log \left| \frac{z_2}{z_1} \right||}{|\log \delta_1|} }
\le -\frac32 + \frac{C}{|\log \delta_1|} \le -\frac32 + \frac{\eta_2}{|\log \delta|} 
\le -\frac32 + \frac{\eta_2}{\log |z_2| - 2\log|z_1| },
$$
since $0< \log |z_2| - 2\log|z_1| \le - \log|z_1| \le -\log \delta$.  Finally the first term of 
the right hand side of
\eqref{threecircineq} is bounded above by 
$$
-\frac32 \left( \log |z_2| - 2\log|z_1|\right) +\eta_2.
$$
For the second term, note first that if $\left( 2 \log \left| \frac{z_2}{z_1} \right| + \eta_1 \right)>0$,
we can just bound the second term by that quantity.  Otherwise,
for $|z_2|<|z_1|$ and 
$\delta_1$ small enough (depending only on $\delta$)
\begin{multline*}
\frac{\log|z_1|- \log \delta_1}{\log \left| \frac{z_2}{z_1} \right|- \log \delta_1}  
=
\frac{1-\frac{\left|\log  |z_1 |\right|}{|\log \delta_1|} }{1-\frac{|\log \left| \frac{z_2}{z_1} \right||}{|\log \delta_1|} }
\ge
1-\frac{\left|\log  |z_1 |\right|}{|\log \delta_1|}
\ge 
\\
1 - \frac{\eta_2}{2 \left| \log \delta\right|}
\ge
1 - \frac{\eta_2}{2 \left| \log \left| \frac{z_2}{z_1} \right| \right|}
\ge
1 + \frac{\eta_2}{2  \log \left| \frac{z_2}{z_1} \right| +\eta_1},
\end{multline*}
and therefore
the second term verifies
$$
\frac{\log|z_1|- \log \delta_1}{\log \left| \frac{z_2}{z_1} \right|- \log \delta_1}  
\left( 2 \log \left| \frac{z_2}{z_1} \right| + \eta_1 \right) 
\le
2  \log \left| \frac{z_2}{z_1} \right| +\eta_1 +\eta_2.
$$
For $|z_1|=|z_2|$, this inequality is immediate
Putting together the two estimates, for $\delta_1$ small enough (depending on $\delta, \eta_2$),
and for $|\eps|$ small enough (depending  on $\delta_1, \eta_1$),
$$
G_{\eps} (z_1,z_2)= u(z_1) \le \log|z_1| + \frac12 \log|z_2| +\eta_1+2\eta_2,
$$
which yields the upper estimate from Lemma \ref{3sym}.

Case 2. $|z_2|^2 \le |z_1|  \le  |z_2| $.

The same proof works, exchanging the roles of the coordinates.

Case 3. $|z_1|^2 \le |z_1+z_2|  \le  \frac12 |z_1| $.

Using the same map $\varphi $, defined only for $|\zeta| < 2/3$, 
computations similar to those of Case 1 (where $z_2$ is replaced by $z_1+z_2$)
yield the desired result. 
\end{proof*}

\subsection{Upper estimates for the degenerate case}

From now on, $S_\eps$ and $G_{\eps}$ are as in Lemma \ref{3sing}.

\subsubsection{Upper estimate for the most common case}

\begin{lemma}
\label{above3sing}
Let $\delta >0$.
Let $(z_1,z_2) \in \D^2 \setminus (D_0 \cup D(0,\delta)^2)$. 
Let $s$ be a function in a neighborhood of $0$ is $\C$ such
that $\lim_{\zeta\to0} s(\zeta)=0$.  

Then there exists $m>0$ such that for any $\alpha \in D(0,m)$, and  
for any $\eta>0$ there exists $\varphi \in \mathcal O(\D,\D^2)$ and
$\zeta_1, \zeta_2 \in D(0,\eta)$, $\zeta_3 \in \D $ such that 
$\left| |\zeta_3|^2 - \max (|z_1|,|z_2|) \right| \le \eta$ and
$$
\varphi(0) = (0,0), 
\varphi (\zeta_1) = ([\alpha s(\alpha)]^2 , 0),
\varphi (\zeta_2) = (0,\alpha^2),
\varphi (\zeta_3) = (z_1,z_2).
$$
\end{lemma}

\begin{proof}

We proceed as in the proof of Lemma \ref{above3sym}, with a few changes.

Case 1. $|z_2|^2 \le |z_1|  \le |z_2| $.

Let $\xi_1, \xi_2 \in D(0,\frac12) \subset \C$. 
We define the mapping $\tilde \varphi_\xi $ from $\D$ to $\C^2$ as before, but
 $ \tilde \zeta_1,  \tilde \zeta_2$ are defined by
$$
\mu := \alpha \left( \frac1{1+\xi_2} + \frac{s(\alpha)^2}{1+\xi_1} \frac{z_2}{z_1} \right)^{1/2},
$$
$$
\tilde \zeta_1 := \frac{\alpha^2 s(\alpha)^2}{(1+\xi_1)\mu} \frac{z_2}{z_1} , \quad
\tilde \zeta_2 :=\tilde \zeta_1 - \mu = -  \frac{\alpha^2}{(1+\xi_2)\mu}.
$$
This ensures that $\tilde \varphi_\xi $ takes the required values at $ \tilde \zeta_1$ and $  \tilde \zeta_2$.

As before, $ \left| \frac{z_2}{z_1} \right| \le \frac1\delta $,
so that for $\xi_1, \xi_2$ small enough and $|\alpha|\le m(\delta)$ ,
$\left| \frac{s(\alpha)^2}{1+\xi_1} \frac{z_2}{z_1}\right| \le \frac16$, thus
$$
\frac12 |\alpha| \le |\mu| \le 2 |\alpha| .
$$
The only point that remains to be modified in the previous proof is that 
$$
\Phi_2(\xi, \alpha):= - \frac{(1+\xi_2)\alpha^2s(\alpha)^2}{(1+\xi_1)\mu} \frac{z_2^{1/2}}{z_1},
$$
which can only improve the contractivity of $\Phi$. 

Case 2. $|z_1|^2 \le |z_2|  \le |z_1| $.

There we set
$$
\mu := \alpha \left( \frac{z_1}{z_2}\frac1{1+\xi_2} + \frac{s(\alpha)^2}{1+\xi_1}  \right)^{1/2},
$$
 $$
\tilde \zeta_1 := -\frac{\alpha^2s(\alpha)^2}{(1+\xi_1)\mu}  , \quad
\tilde \zeta_2 :=\tilde \zeta_1 + \mu = -  \frac{\alpha^2}{(1+\xi_2)\mu}\frac{z_1}{z_2}.
$$
Again, we see that $1\le \left| \frac{z_1}{z_2} \right| \le \frac1\delta $,
so for $\alpha$ small enough,
$$
\frac1{\sqrt{2}} |\alpha| \le |\mu| \le \frac2{\sqrt{\delta}} |\alpha|.
$$
The proof proceeds as before, with the single change
$$
\Phi_2(\xi, \alpha) :=  \frac{(1+\xi_2)\alpha^2s(\alpha)^2}{(1+\xi_1)\mu} \frac{1}{z_1^{1/2}}.
$$
\end{proof}

The following statement is proved from the previous Lemma exactly as Lemma \ref{roughestabove}, with a simpler proof as we don't need to consider the values
of $\frac{z_1}{z_2} +1$. 

\begin{cor}
\label{roughestsing}
Under the hypotheses of Lemma \ref{above3sing},
 there exists $C>0$ such that for any
$\delta >0$, there is an $m=m(\delta)>0$ such that for any $\eps \in D(0,m)$, 
 for any $(z_1,z_2) \in \D^2 \setminus (D_0 \cup D(0,\delta)^2)$,
$$
G_{\eps} (z_1,z_2) \le \frac32 \log \left( \max (|z_1|, |z_2|) \right) + C .
$$
\end{cor}

\subsubsection{Upper estimate near the coordinate axes}

Now we need to estimate the (faster) decrease of the Green function on the 
exceptional region $\check D'_0 := 
\left\{ z \in \D^2 \setminus \{0\} : |z_1| \le |z_2|^{3/2} 
\mbox{ or } |z_2| \le |z_1|^2 \right\}$.  

\begin{lemma}
\label{aboveD'0}
Let $\delta >0$. For any $\eta>0$
there exists $m=m(\delta, \eta)>0$ such that for any $\eps \in D(0,m)$,
for any $(z_1,z_2) \in \check D'_0 \setminus D(0,\delta)^2$,
$$
G_{\eps} (z_1,z_2) \le 2 \log \left( \max (|z_1|, |z_2|) \right) + \eta.
$$
\end{lemma}

\begin{proof}
In the cases  $ |z_2|  \le |z_1|^2 $ or 
$ |z_1|  \le |z_2|^2 $ , the  proof of Lemma \ref{aboveD0} will apply.  

When $|z_2|^2 < |z_1| \le |z_2|^{3/2}$, some new trick is required.  This time, the analytic disk 
we construct will have to pass twice through one of the poles; this should be compared
to Poletsky's theorem \cite{Po-Sh}, \cite{Po}, see also \cite{La-Si},
 where the Green function is recovered from analytic disks 
that may have to hit the pole(s) more than once. Our disks will be perturbations of
the Neil parabola $\zeta \mapsto (\zeta^3, \zeta^2)$.

We write $s(\eps) = \rho(\eps)/\eps = o(1)$.

Choose complex numbers $\lambda, \mu$ such that
$$
\lambda^2 := \frac{z_1}{z_2(z_2-\eps)}
\left( \frac{z_1}{z_2-\eps} + s(\eps) \right) ; \quad
\mu^2 := \eps + \left( \frac{s(\eps)}{2 \lambda} \right)^2.
$$
Let
$$
\Psi_{\lambda, \mu}(\zeta) :=
\left(  \left( \lambda \zeta - \frac12 s(\eps) \right) (\zeta^2-\mu^2),
\zeta^2 -  \left( \frac{s(\eps)}{2 \lambda} \right)^2 
\right) .
$$
Then by construction
$$
\Psi_{\lambda, \mu}(\mu) = \Psi_{\lambda, \mu}(-\mu)=(0,\eps) ,
\Psi_{\lambda, \mu} ( \frac{s(\eps)}{2 \lambda}  ) = (0,0) ,
\Psi_{\lambda, \mu} ( -\frac{s(\eps)}{2 \lambda}  ) = (\eps s(\eps), 0),
$$
so we have a disk passing through all three poles of $G_{\eps}$.
Furthermore, choosing 
$$
\zeta_z := \frac1\lambda \left( \frac{z_1}{z_2-\eps} + s(\eps) \right),
$$
we have $\Psi_{\lambda, \mu}(\zeta_z)= z$. Notice that
$$
\zeta_z^2 = \frac{z_2(z_2-\eps)}{z_1} \left( \frac{z_1}{z_2-\eps} + s(\eps) \right),
$$
so there is some $\eps_0(\delta, \eta)>0$ such that for $|\eps|< \eps_0(\delta, \eta)$,
for any $z \notin D(0,\delta)^2$ such that $|z_2|^2 < |z_1| \le |z_2|^{3/2}$, 
\begin{equation}
\label{modzeta}
\left| |\zeta_z| - |z_2|^{1/2}\right| \le \eta. 
\end{equation}
In particular,
if $z$ remains in a compact subset of $\D^2$ avoiding the origin, by choosing $\eta$
small enough we ensure that $\zeta \in \D$.  We need a more general fact.

Claim.  

Let $\eta>0$, and $\delta >0$.  Then there exists 
 $\eps_1 = \eps_1(\delta, \eta)>0$ such that for any $\eps$
with $|\eps|\le \eps_1$, for any
 $z \in \D^2 \setminus D(0,\delta)^2$ such that 
$|z_2|^2 < |z_1| \le |z_2|^{3/2}$, we have $\Psi_{\lambda, \mu} (D(0,1-\eta) )
\subset \D^2$. 

Proof of the Claim.

For $|\eps|\le \delta/2$, $|z_2|/2 \le |z_2-\eps| \le 2|z_2|$, so
$$
|\lambda|^2 \ge \left| \frac{z_1}{2z_2^2} \right| 
\left(  \left| \frac{z_1}{2z_2} \right| - |s(\eps)| \right) 
\ge \left| \frac{z_1^2}{8z_2^3} \right| \ge \frac\delta8,
$$
for $\eps$ small enough. So when $|\zeta|\le 1-\eta$, 
$$
\left| \Psi_{\lambda, \mu,2}(\zeta)\right| 
\le (1-\eta)^2 + \frac{2 |s(\eps)|^2}\delta <1
$$
for $\eps$ small enough. 

In a similar way, given $\eta'$, 
for $\eps$ small enough depending on $\delta$ and $\eta'$, we have
$|z_2| \le (1+\eta') |z_2-\eps|$, so
$$
|\lambda|^2 \le (1+\eta')^2 \left| \frac{z_1}{z_2^2} \right| 
\left(  \left| \frac{z_1}{z_2} \right| +  \frac{|s(\eps)|}{(1+\eta') } \right) 
\le (1+\eta')^3 \left| \frac{z_1^2}{z_2^3} \right| \le (1+\eta')^3
$$
for  $\eps$ small enough. Choose $\eta'$ so that 
$(1+\eta')^3 = (1+\eta)$.
When $|\zeta|\le 1-\eta$, 
$$
\left| \Psi_{\lambda, \mu,1}(\zeta)\right| 
\le
\left(  (1+\eta)  (1-\eta) + \frac12 |s(\eps)| \right) 
\left( (1-\eta)^2 + |\eps| +  \frac{4|s(\eps)|^2 }{\delta} \right) <1
$$
for  $\eps$ small enough.  \hfill \qed

So now the function $v(\zeta):= G_{\eps} \left( \Psi_{\lambda, \mu}((1-\eta)\zeta) \right) $
is negative and subharmonic on $\D$.  Furthermore, it has logarithmic poles at the points
$\pm \frac\mu{1-\eta}$ and $\pm  \frac{s(\eps)}{2 \lambda(1-\eta)} $ ; in the cases when $\mu=0$ or
$s(\eps)=0$, we get a double logarithmic pole at the corresponding point. 

Denote $d_{\D}(\zeta, \xi):= \left| \frac{\zeta-\xi}{1-\zeta \bar \xi} \right|$. Then
\begin{multline*}
G_{\eps} (z) = v(\zeta_z) \le
\log d_{\D}(\zeta_z, \frac\mu{1-\eta}) + \log d_{\D}(\zeta_z, -\frac\mu{1-\eta}) \\
+ \log d_{\D}(\zeta_z, \frac{s(\eps)}{2 \lambda (1-\eta)})  + 
\log d_{\D}(\zeta_z, -\frac{s(\eps)}{2 \lambda (1-\eta)}). 
\end{multline*}
By \eqref{modzeta}, choosing
$m(\delta, \eta)$ accordingly, we have, for $|\eps|\le m$,
$G_{\eps} (z) \le 4 \log |z_2|^{1/2} +O(\eta)$. Changing 
the value of $\eta$, we have the conclusion.  
\end{proof}

\subsubsection{End of proof of Lemma \ref{3sing}}
\begin{proof*}

In the case where $|z_2|\le |z_1|$, the proof goes exactly as for Lemma \ref{3sym}
(without having to take into account the value of $\frac{z_2}{z_1}+1$).

When $|z_1|\le |z_2|$, it is enough to consider the case where $|z_1|\ge |z_2|^{3/2}$
(Lemma \ref{aboveD'0} takes care of the remaining case). 

Consider the analytic disk
$$
\varphi (\zeta) = \left( \frac{z_1}{z_2} \zeta, \zeta \right),
\mbox{ and } u(\zeta):=  G_{\eps} \left(\varphi(\zeta) \right) .
$$
Let $\delta_1<\delta$, to be chosen later. By Corollary \ref{roughestsing},
for $|\zeta| = \delta_1$,
$$
u(\zeta)=  G_{\eps} \left(\varphi(\zeta) \right) \le \frac32 \log \delta_1 +C.
$$
On the other hand, if 
$|\zeta|= \left| \frac{z_1}{z_2} \right|^2$, then $\varphi (\zeta) \in
\{ w : |w_1|\le |w_2|^{3/2}\}$ and by Lemma \ref{aboveD'0}, for any $\eta_1>0$ we can choose $|\eps|$ small
enough so that 
$$
u(\zeta) \le 4 \log \left| \frac{z_1}{z_2} \right| + \eta_1.
$$
Applying the three-circle theorem for $\zeta = z_2$, we get 
\begin{multline}
\label{3circsing}
G_{\eps} (z) = u(z_2) \le 
\frac{2 \log \left| \frac{z_1}{z_2} \right| - \log|z_2|}{2 \log \left| \frac{z_1}{z_2} \right| - \log \delta_1}  
\left(  \frac32 \log \delta_1 +C \right)  \\
+
\frac{\log|z_2|- \log \delta_1}{2 \log \left| \frac{z_1}{z_2} \right|- \log \delta_1}  
\left( 4 \log \left| \frac{z_1}{z_2} \right| + \eta_1 \right) .
\end{multline}
Let $\eta_2>0$. In order to estimate the first term in the right hand side of  the above inequality, we use
$$
\frac{\frac32 \log \delta_1 +C }{2 \log \left| \frac{z_1}{z_2} \right| - \log \delta_1}  
=
 \frac{-\frac32 + \frac{C}{|\log \delta_1|}}{1+\frac{2 \log \left| \frac{z_1}{z_2} \right|}{|\log \delta_1|} }
\le -\frac32 + \frac{C}{|\log \delta_1|},
$$
for $\delta_1 $ small enough.   Notice that $2 \log \left| \frac{z_1}{z_2} \right| - \log|z_2|=
\log \left| \frac{z_1^2}{z_2^3} \right| \ge 0$, and also $\log \left| \frac{z_1^2}{z_2^3}\right| \le 
- \log |z_2|\le -\log \delta$, so
$$
-\frac32 + \frac{C}{|\log \delta_1|} 
\le -\frac32 + \frac{\eta_2}{|\log \delta|}
\le  -\frac32 + \frac{\eta_2}{\log \left| \frac{z_1^2}{z_2^3} \right|}.
$$
To estimate the second term in the right hand side of \eqref{3circsing}, it is
enough to consider the case where $4 \log \left| \frac{z_1}{z_2} \right| + \eta_1<0$.
We use
\begin{multline*}
\frac{\log|z_2|- \log \delta_1}{2 \log  \left| \frac{z_1}{z_2} \right|- \log \delta_1}  
=
\frac{1-\frac{\left|\log  |z_2 |\right|}{|\log \delta_1|} }{1-\frac{2|\log \left| \frac{z_1}{z_2} \right||}{|\log \delta_1|} }
\ge
1-\frac{\left|\log  |z_2 |\right|}{|\log \delta_1|}
\ge 
\\
1 - \frac{\eta_2}{ \left| 4 \log \delta + \eta_1\right|}
\ge
1 - \frac{\eta_2}{ \left| 4 \log \left| \frac{z_1}{z_2} \right| + \eta_1\right|},
\end{multline*} 
for $\delta_1$ small enough. 

Finally, \eqref{3circsing} implies 
$$
u(z_2) \le  -\frac32 
\left( 2 \log \left| \frac{z_1}{z_2} \right| - \log|z_2| \right) + \eta_2 
+ 4 \log \left| \frac{z_1}{z_2} \right| +\eta_1 + \eta_2,
$$
from which the desired estimate follows. 
\end{proof*}

\end{document}